\documentclass[12pt]{article}

\usepackage{graphicx} 
\usepackage{geometry} 
\usepackage{mathtools,amsmath,amssymb}
\usepackage{amsthm}
\usepackage[ruled, linesnumbered]{algorithm2e}
\usepackage{dsfont}
\usepackage{microtype}
\usepackage{enumerate}
\usepackage{subcaption}
\usepackage{nicefrac}
\usepackage{xcolor}

\definecolor{color1}{HTML}{d10050}
\definecolor{color2}{HTML}{d1ab00}
\definecolor{color3}{HTML}{69d100}
\definecolor{color4}{HTML}{00d176}
\definecolor{color5}{HTML}{00aed1}
\definecolor{color6}{HTML}{000ed1}
\definecolor{color7}{HTML}{8400d1}

\definecolor{cbcolor1}{HTML}{D81B60}
\definecolor{cbcolor2}{HTML}{1E88E5}
\definecolor{cbcolor3}{HTML}{FFC107}
\definecolor{cbcolor4}{HTML}{004D40}

\usepackage[colorlinks,
linkcolor=color6,        
citecolor=color6,         
filecolor=color6,     
urlcolor=color6,
breaklinks=true]{hyperref}
\usepackage[noabbrev,capitalise,english]{cleveref}
\crefname{conjecture}{conjecture}{conjectures}
\usepackage{todonotes}
\usepackage{array}
\usepackage{multirow}
\usepackage{bm}

\newtheorem{observation}{Observation}
\newtheorem{conjecture}{Conjecture}
\newtheorem{theorem}{Theorem}
\newtheorem{definition}{Definition}
\newtheorem{lemma}{Lemma}
\newtheorem{corollary}{Corollary}

\theoremstyle{remark}
\newtheorem{example}{Example}
\newtheorem*{remark}{Remark}
\newtheorem*{claim}{Claim}

\title{Integer and Unsplittable Multiflows\\ in Series-Parallel Digraphs%
\thanks{Funded by the Deutsche Forschungsgemeinschaft (DFG, German Research Foundation) under Germany's Excellence Strategy --- The Berlin Mathematics Research Center MATH+ (EXC-2046/1, project ID: 390685689).}
}
\author{
    Mohammed Majthoub Almoghrabi\\
    Technische Universität Berlin, Germany\\
    \texttt{majthoub@math.tu-berlin.de}
    \and
    Martin Skutella\\
    Technische Universität Berlin, Germany\\
    \texttt{skutella@math.tu-berlin.de}
    \and
    Philipp Warode\\
    Humboldt-Universität zu Berlin, Germany\\
    \texttt{philipp.warode@hu-berlin.de}
}
\date{\today}

\begin{document}
\maketitle

\begin{abstract}
An unsplittable multiflow routes the demand of each commodity along a single path from its source to its sink node. As our main result, we prove that in series-parallel digraphs, any given multiflow can be expressed as a convex combination of unsplittable multiflows, where the total flow on any arc deviates from the given flow by less than the maximum demand of any commodity.
This result confirms a 25-year-old conjecture by Goemans for single-source unsplittable flows, as well as a stronger recent conjecture by Morell and Skutella, for series-parallel digraphs---even for general multiflow instances where commodities have distinct source and sink nodes.
Previously, no non-trivial class of digraphs was known for which either conjecture holds.
En route to proving this result, we also establish strong integrality results for multiflows on series-parallel digraphs, showing that their computation can be reduced to a simple single-commodity network flow problem.
\end{abstract}

\usetikzlibrary{decorations.pathreplacing,decorations.shapes,decorations.pathmorphing,shapes}
\tikzstyle{CircleLight}=[draw, thick, circle, inner sep=0pt, minimum size = 3ex, fill=gray!15]
\tikzstyle{CircleLightSmall}=[draw, thick, circle, inner sep=0pt, minimum size = 2ex, fill=gray!15]
\tikzstyle{CircleDark}=[draw, thick, circle, inner sep=0pt, minimum size = 3ex, fill=gray!15]
\tikzstyle{PNode}=[draw, thick, circle, inner sep=0pt, minimum size = 3.5ex, fill=cbcolor1!65]
\tikzstyle{SNode}=[draw, thick, circle, inner sep=0pt, minimum size = 3.5ex, fill=cbcolor2!65]
\tikzstyle{QNode}=[draw, thick, circle, inner sep=0pt, minimum size = 3.5ex, fill=gray!15]

\newcommand{\colorbrace}[3]{\textcolor{#1}{\underbrace{\color{black}#2}_{#3}}}
\renewcommand{\vec}[1]{\boldsymbol{\mathbf{#1}}}
\newcommand{\R}{\mathbb{R}}
\newcommand{\N}{\mathbb{N}}
\newcommand{\Z}{\mathbb{Z}}
\newcommand{\lexle}{\prec_{\text{LEX}}}
\newcommand{\lexleq}{\preceq_{\text{LEX}}}
\newcommand{\indicator}{\mathds{1}}
\newcommand{\spnode}{\omega}
\newcommand{\yshare}{s}
\newcommand{\mcflow}{X}
\newcommand{\totalflow}{x}
\newcommand{\componentflow}[1][x]{\bar{#1}}
\newcommand{\innerV}{V^{\circ}}
\newcommand{\fractionalDemands}{I}
\newcommand{\unsplittableDemands}{\bar{I}}
\newcommand{\demandSet}{J}
\newcommand{\ileft}[1][\spnode]{p^{#1}}
\newcommand{\iright}[1][\spnode]{q^{#1}}
\newcommand{\imiddle}[1][\spnode]{r^{#1}}

\newcommand{\flowgroup}{K}
\newcommand{\numflows}{L}

\newcommand{\smax}{\mathop{\mathrm{max}_2}}

\newcommand{\argminA}{\mathop{\mathrm{arg\,min}}}


\setcounter{page}{0}
\newpage
\section{Introduction}

Network flow theory focuses on the optimal routing of commodities through networks and is fundamental to a wide array of applications in fields such as logistics, telecommunications, and computer networking; see, e.g.,~\cite{AhujaMagOrl93,Williamson-book2019}. In many scenarios, splitting a commodity might degrade the quality of service, increase complexity, or is simply infeasible. For example, in optical networks, dividing data streams across multiple paths requires specialized equipment. Similarly, in freight logistics, splitting loads across multiple routes is often impractical.

This motivates \emph{unsplittable} multiflows. Here, for a given digraph with~$k$ source-sink pairs, the given demand~$d_i$ of every commodity~$i\in[k]\coloneqq\{1,\dots,k\}$ must be sent along a single path~$P_i$ from its source~$s_i$ to its sink~$t_i$. To differentiate unsplittable flows from traditional flows, where the demand of each commodity may be divided and sent along multiple $s_i$--$t_i$~paths, we refer to the latter as \emph{fractional} flows.

Unsplittable flow problems represent a compelling extension of disjoint path problems and have received considerable attention in the literature. Even in the apparently straightforward case of a (series-parallel) digraph consisting of one common source and one common sink connected by parallel arcs, determining the existence of a \emph{feasible} unsplittable multiflow (i.e., one that adheres to given arc capacities) is NP-complete. In fact, several classical problems in combinatorial optimization, such as, e.g., Bin Packing, Partition, or parallel machine scheduling with makespan objective occur as special cases; for more details we refer to Kleinberg's PhD thesis~\cite{Kleinberg-Diss96}.

\paragraph*{Integer multiflows.}
When every commodity~$i\in[k]$ has unit demand~$d_i\coloneqq1$, unsplittable multiflows precisely correspond to integer multiflows, which have been extensively studied in the literature. In fact, in unit-capacity networks, feasible integer multiflows model arc-disjoint paths, and therefore cannot be computed efficiently, unless P=NP. For example, Vygen~\cite{Vygen1995} shows that the arc-disjoint paths problems is NP-complete even in acyclic and planar digraphs; see~\cite[Chapter 70.13b+c]{Schrijver03}  for further related complexity results. On the other hand, single-commodity network flows (i.e., $k=1$) are well-known to feature strong integrality properties, and there are efficient combinatorial algorithms to compute integer single-commodity flows; see, e.g.,~\cite{AhujaMagOrl93,Williamson-book2019}. In particular, the network flow polytope is integral, that is, any single-commodity flow~$(x_e)_{e\in E}$ on a digraph~$G=(V,E)$ with integer excess at every node, is a convex combination of integer flows~$(y_e)_{e\in E}$ that maintain the same excesses, such that
\begin{align}
\label{eq:rounding-bounds}
\lfloor x_e\rfloor\leq y_e\leq\lceil x_e\rceil\qquad\text{for every arc~$e\in E$.}
\end{align}

Arguably, the most important tool for proving the existence of integer multiflows in special graph classes is the Nagamochi--Ibaraki Theorem~\cite{NagamochiIbaraki1989} (see also~\cite[Theorem~70.8]{Schrijver03}), which relies on the sufficiency of the \emph{cut condition}. This (necessary) condition requires that the capacity of every cut exceeds the total demand of commodities that must cross it; see~\eqref{eq:cut-condition} below.
The concept of cut-sufficiency is of interest in its own right and is an active research topic; see, e.g., the recent work of Poremba and Shepherd~\cite{poremba2023cut} and references therein.
For a fixed digraph, the Nagamochi--Ibaraki Theorem asserts that if, for any integer arc capacities and demands, the cut condition guarantees the existence of a feasible multiflow, it also guarantees the existence of a feasible \emph{integer} multiflow. For most classes of digraphs in which the existence of a feasible multiflow implies the existence of a feasible integer multiflow, the cut condition is indeed sufficient. However, the multiflow instance in \Cref{fig:cut-cond} illustrates that the cut condition is generally not sufficient for the class of series-parallel digraphs studied in this paper.
\begin{figure}[tb]
\centering
\begin{tikzpicture}
\node [CircleLight] (s1) at (-2.5,0) {$s_1$};
\node [CircleDark] (s2) at (-1,0) {$s_2$};
\node [CircleLight] (t1) at (1,0) {$t_1$};
\node [CircleDark] (t2) at (2.5,0) {$t_2$};
\draw [->,thick] (s1) -- (s2);
\draw [->,thick] (s1) to [out=20,in=160] (t2);
\draw [->,thick] (s2) -- (t1);
\draw [->,thick] (t1) -- (t2);
\end{tikzpicture}
\caption{An infeasible multiflow instance with two unit-demand commodities on a series-parallel digraph with unit arc capacities satisfying the cut condition}
\label{fig:cut-cond}
\end{figure}
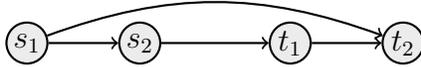

\paragraph*{Single-source unsplittable flows.}
Even though our results hold for general multiflow instances in series-parallel digraphs, where each commodity is routed from a distinct source to a distinct sink, our main result is primarily inspired by prior research on single-source unsplittable flows, which we review next. For a comprehensive overview of results on general unsplittable multiflows, we refer to Kolliopoulos' survey~\cite{kolliopoulos2007edge}.

Single-source unsplittable flows, where all commodities share a common source node~$s$, have first been studied by Kleinberg~\cite{Kleinberg-Diss96}. Dinitz, Garg, and Goemans~\cite{DGG99} prove that a given fractional flow~$(x_e)_{e\in E}$ can always be turned into an unsplittable flow~$(y_e)_{e\in E}$ (given by~$s$--$t_i$~paths~$P_i$ with~$y_e\coloneqq\sum_{i:e\in P_i}d_i$) such that
\begin{align}
\label{eq:DGG}
y_e\leq x_e+d_{\max}\qquad\text{for every arc~$e\in E$,}
\end{align}
where~$d_{\max}\coloneqq\max_{i\in[k]}d_i$. A famous conjecture of Goemans says that flow~$(x_e)_{e\in E}$ can even be expressed as a convex combination of unsplittable flows~$(y_e)_{e\in E}$ that satisfy~\eqref{eq:DGG}.\footnote{Goemans' original conjecture asserts that, for any arc costs, an unsplittable flow can be found that satisfies~\eqref{eq:DGG} and whose cost does not exceed that of~$(x_e)_{e\in E}$. This is equivalent to the stated existence of a convex combination; see, e.g.,~\cite{MartensSalazarSkut06}.} Skutella~\cite{Skut-MathProg2002}, based on a flow augmentation method also used by Kolliopoulos and Stein~\cite{KolliopoulosStein2002}, proves that Goemans' Conjecture is valid when the demands of the commodities are multiples of one another. For acyclic digraphs, Morell and Skutella~\cite{MorellSkutella2021} show that any fractional flow~$(x_e)_{e\in E}$ can be turned into an unsplittable flow~$(y_e)_{e\in E}$ that satisfies the lower bound
\begin{align}
\label{eq:MS}
y_e\geq x_e-d_{\max}\qquad\text{for every arc~$e\in E$.}
\end{align}
This result can also be achieved by augmenting flow in the reverse direction along the designated cycles of Dinitz, Garg, and Goemans~\cite{DGG99}. Furthermore, they conjecture the existence of an unsplittable flow that satisfies both the upper bounds~\eqref{eq:DGG} and the lower bounds~\eqref{eq:MS}. Only recently, for the special case of acyclic and planar digraphs, this conjecture has been proved by Traub, Vargas Koch, and Zenklusen~\cite{TraubVargasKochZenklusen2024}. Morell and Skutella also propose the following strengthening to Goemans' conjecture.

\begin{conjecture}[Morell and Skutella~\cite{MorellSkutella2021}]
\label{conj:strong}
For the single-source unsplittable flow problem on acyclic digraphs, any fractional flow~$(x_e)_{e\in E}$ can be expressed as a convex combination of unsplittable flows~$(y_e)_{e\in E}$ that satisfy both~\eqref{eq:DGG} and~\eqref{eq:MS}.
\end{conjecture}

Using techniques of Martens, Salazar, and Skutella~\cite{MartensSalazarSkut06}, they show that their \Cref{conj:strong} is valid when the demands of the commodities are multiples of one another. Moreover, for the special case of acyclic and planar digraphs, and under the slightly relaxed lower and upper bounds~$x_e-2d_{\max}\leq y_e\leq x_e+2d_{\max}$ for~$e\in E$, \Cref{conj:strong} is proved by Traub, Vargas Koch, and Zenklusen~\cite{TraubVargasKochZenklusen2024}, based on a reduction to a well-structured discrepancy problem.

\paragraph*{Contribution and outline}
After discussing preliminaries and notation for multiflows on series-parallel digraphs\footnote{We emphasize that we only consider series-parallel \emph{directed} graphs in this paper and, thus, do not contribute to the line of research on multiflows in (series-parallel) undirected graphs; see, e.g.,~\cite{chekuri2013flow}.} in \Cref{sec:preliminaries}, in \Cref{sec:cut-condition} we present the following observations on (integer) multiflows in series-parallel digraphs.

\begin{theorem}
\label{thm:integrality}
Consider a multiflow instance on a series-parallel digraph.
\begin{enumerate}[(a)]
\item\label{thm:integrality:a} For integer demands,  the total arc flows~$(x_e)_{e\in E}$ of any multiflow are a convex combination of  total arc flows~$(y_e)_{e\in E}$ of integer multiflows  satisfying~\eqref{eq:rounding-bounds}.
\item\label{thm:integrality:b} For (integer) arc capacities, a feasible (integer) multiflow, if one exists, can be efficiently found through a single-commodity flow computation.
\end{enumerate}
\end{theorem}

\begin{remark}
Notice that \cref{thm:integrality}(a) is only concerned with the total arc flows of multiflows instead of the multiflows itself.  In fact, as stated in the theorem, every \emph{total flow} of a multiflow can be expressed as a convex combination of \emph{total flows} of integer multiflows satisfying~\eqref{eq:rounding-bounds}. However, this is in general not true for the multiflow vector, i.e., the vector containing all commodity flows.  In \cref{fig:counter-example_CONV-MUSF} we give an example of a multiflow that cannot be expressed as a convex combination of integer multiflows, while its total flow can.
\end{remark}

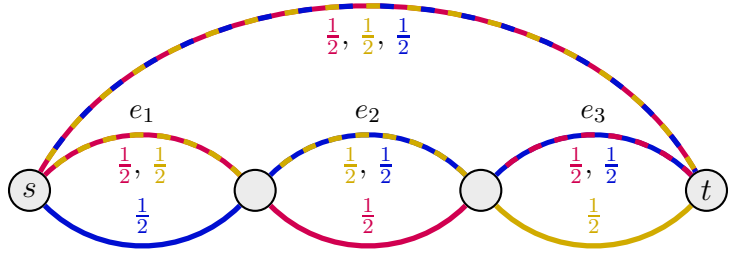
\begin{figure}[t]
    \centering
    \begin{tikzpicture}
        \node[CircleLight] (A) at (0,0) {$s$};
        \node[CircleLight] (B) at (3,0) {};
        \node[CircleLight] (C) at (6,0) {};
        \node[CircleLight] (D) at (9,0) {$t$};

        \draw[cbcolor2, line width=2pt, bend right=35] (A) to  node[midway, above] {\small $\nicefrac{1}{2}$} (B);
        \draw[cbcolor1, line width=1.9pt, bend left=35] (A) edge node[midway, above,black] {\small $e_1$}  node[midway, below,black] {\small \textcolor{cbcolor1}{$\nicefrac{1}{2}$},  \textcolor{cbcolor3}{$\nicefrac{1}{2}$}} (B);
        \draw[cbcolor3, line width=2pt, bend left=35,dash pattern=on 0ex off 1ex on 1ex off 0ex] (A) edge (B);

        \draw[cbcolor1, line width=2pt, bend right=35] (B) to node[midway, above] {\small $\nicefrac{1}{2}$} (C);
        \draw[cbcolor2, line width=1.9pt, bend left=35] (B) edge node[midway, above,black] {\small $e_2$}  node[midway, below,black] {\small  \textcolor{cbcolor3}{$\nicefrac{1}{2}$},  \textcolor{cbcolor2}{$\nicefrac{1}{2}$}} (C);
        \draw[cbcolor3, line width=2pt, bend left=35,dash pattern=on 0ex off 1ex on 1ex off 0ex] (B) edge (C);

        \draw[cbcolor3, line width=2pt, bend right=35] (C) to node[midway, above] {\small $\nicefrac{1}{2}$} (D);
        \draw[cbcolor2, line width=1.9pt, bend left=35] (C) edge  node[midway, above,black] {\small $e_3$}  node[midway, below,black] {\small \textcolor{cbcolor1}{$\nicefrac{1}{2}$},  \textcolor{cbcolor2}{$\nicefrac{1}{2}$}} (D);
        \draw[cbcolor1, line width=2pt, bend left=35,dash pattern=on 0ex off 1ex on 1ex off 0ex] (C) edge (D);

        \draw[cbcolor1, line width=1.9pt, bend left=35, out=50, in=130] (A) to (D);
        \draw[cbcolor3, line width=2pt, bend left=35, out=50, in=130, dash pattern=on 0ex off 1ex on 1ex off 1ex] (A) to node[midway, below,black] {\small \textcolor{cbcolor1}{$\nicefrac{1}{2}$},  \textcolor{cbcolor3}{$\nicefrac{1}{2}$},  \textcolor{cbcolor2}{$\nicefrac{1}{2}$}} (D);
        \draw[cbcolor2, line width=2pt, bend left=35, out=50, in=130, dash pattern=on 0ex off 2ex on 1ex off 0ex] (A) to (D);
    \end{tikzpicture}
    \caption{An instance of a multiflow that cannot be expressed as the convex combination of integer multiflows that satisfy~\eqref{eq:rounding-bounds}. In this instance, three commodities (red, yellow, blue) route a demand of~$1$ from $s$ to $t$ each.  The fractional flow is given by the numbers above.  In particular, the fractional multiflow routes a flow of $\nicefrac{1}{2}$ of exactly two commodities over each of the arcs $e_1, e_2, e_3$ resulting in a total flow of $x_e = 1$ on these arcs. By~\eqref{eq:rounding-bounds}, every integer flow must therefore also send a total flow of $y_e = 1$ over these arcs which is not possible since the integer flow can only be sent over arcs with fractional flow of the same commodity.}
    \label{fig:counter-example_CONV-MUSF}
\end{figure}

In light of the fact that multiflow instances on series-parallel digraphs generally do not satisfy the assumptions of the Nagamochi--Ibaraki Theorem (i.e., the cut condition is not sufficient; see \Cref{fig:cut-cond}), the strong integrality property in \Cref{thm:integrality}\eqref{thm:integrality:a} might seem surprising. On the other hand, the proof of \Cref{thm:integrality} relies on the simple observation that, by carefully subdividing commodities, multiflow instances on series-parallel digraphs can be efficiently reduced to a certain subclass of instances which can be solved by single-commodity flow techniques. The integrality property in \Cref{thm:integrality}\eqref{thm:integrality:a} is then inherited from the integrality of the network flow polytope.
To the best of our knowledge, and somewhat surprisingly, these observations have not appeared in the literature before. The general idea of subdividing commodities, however, was used before, for example, in the proof of the Okamura-Seymour Theorem~\cite{OkamuraSeymour1981}.

At the end of \Cref{sec:cut-condition}, we discuss two strengthenings of the classical cut condition that relies on considering arbitrary multicuts for subsets of commodities. We prove that this stronger cut condition is necessary and sufficient for the existence of feasible multiflows on series-parallel digraphs.

In \Cref{sec:almost-unsplittable}, as an intermediate step toward constructing unsplittable multiflows, we introduce the notion of \emph{almost unsplittable} multiflows. In such multiflows, at most two commodities are allowed to be routed fractionally through any subdigraph that arises in the series-parallel composition of a series-parallel digraph. We show that any fractional multiflow can be transformed into an almost unsplittable multiflow without altering the total flow values on the arcs.

Finally, in \Cref{sec:convex:combination} we obtain our main result, which generalizes the strong integrality result in \Cref{thm:integrality}\eqref{thm:integrality:a} towards unsplittable multiflows.

\begin{theorem}
\label{thm:main}
The total arc flows~$(x_e)_{e\in E}$ of a fractional multiflow in a series-parallel digraph can be expressed as a convex combination of total arc flows~$(y_e)_{e\in E}$ of unsplittable multiflows that satisfy
\vspace{-0.5ex}
\begin{align*}
\label{eq:main}
x_e-d_{\max}<y_e<x_e+d_{\max}\qquad\text{for every arc~$e\in E$.}
\end{align*}
\end{theorem}

This result implies, in particular, that \Cref{conj:strong} holds for series-parallel digraphs, even for general multiflow instances where commodities have individual source and sink nodes. Even for the weaker conjecture of Goemans,
\Cref{thm:main} provides the first proof for a non-trivial class of digraphs.

\Cref{thm:main} is obtained as follows. We first show that, without altering the total arc flows, any multiflow can be transformed into an \emph{almost unsplittable} multiflow, where no more than two commodities are fractionally routed through any component of the series-parallel digraph. Such an almost unsplittable multiflow can be readily expressed as a convex combination of unsplittable multiflows, with total arc flows deviating by no more than~$2d_{\max}$ (\Cref{thm:2dmax}). Note that for the specific case of single-source multiflows, Traub, Vargas Koch, and Zenklusen~\cite{TraubVargasKochZenklusen2024} prove the same result for the broader class of acyclic and planar digraphs. Reducing the bound from~$2d_{\max}$ to~$d_{\max}$, however, requires more elaborate arguments and refined techniques. It is worthwhile to mention that, while \Cref{thm:main} is formulated as an existence result, these unsplittable multiflows can be efficiently computed.


\section{Preliminaries and notation}
\label{sec:preliminaries}
\paragraph{Multiflows.}
We consider a directed graph~$G=(V,E)$ together with~$k$ commodities given by source-sink pairs~$(s_i,t_i)\in V\times V$, $i\in[k]\coloneqq\{1,\dots,k\}$.
For given demands~$d_i\in\R_{\geq0}$, $i\in[k]$, a multiflow~$\vec{\mcflow}$ sends~$d_i$ units of flow from~$s_i$ to~$t_i$ for every commodity~$i\in[k]$. That is,~$\vec{\mcflow}$ specifies for every commodity~$i\in[k]$ and every arc~$e\in E$ a flow value~$\mcflow_{e,i}\geq0$, such that the following flow conservation constraints hold:
\[
	\sum_{e\in\delta^+(v)} \mcflow_{e, i} - \sum_{e\in\delta^-(v)} \mcflow_{e, i} =
	\begin{cases}
		d_i & \text{if~$v=s_i$,}\\[-0.5ex]
		-d_i & \text{if~$v=t_i$,}\\[-0.5ex]
		0 & \text{otherwise,}
	\end{cases}
	\qquad
	\text{for every~$v\in V$, $i\in[k]$.}
\]
Here,~$\delta^+(v)$ and~$\delta^-(v)$ denote the sets of arcs~$e\in E$ leaving and entering node~$v$, respectively.
In what follows we always assume that~$d_i>0$,~$s_i\neq t_i$, and that there exists an $s_i$--$t_i$~path in~$G$ for every~$i\in[k]$. A multiflow~$\vec{\mcflow}$ is \emph{unsplittable} if the demand~$d_i$ is routed along a single~$s_i$--$t_i$~path, for every~$i\in[k]$. More precisely,~$\mcflow_{e,i}\in\{0,d_i\}$ for every~$i\in[k]$,~$e\in E$, and the arcs~$e\in E$ with~$\mcflow_{e,i}=d_i$ form an~$s_i$--$t_i$~path.

For a multiflow~$\vec{\mcflow}=(\mcflow_{e,i})_{e\in E, i\in[k]}$, the \emph{total flow}~$\vec{\totalflow}=(\totalflow_{e})_{e \in E}$ is given by the sum of the flows of all commodities, that is,
$\totalflow_e\coloneqq\sum_{i\in[k]}\mcflow_{e,i}$. A  multiflow~$\vec{\mcflow}$ satisfies given \emph{arc capacities}~$c_e\in\R_{\geq0}$,~$e\in E$, if~$\totalflow_e\leq c_e$ for every arc~$e\in E$. In this case we say that multiflow~$\vec{X}$ is \emph{feasible}. A multiflow instance is called \emph{feasible} if there is a feasible multiflow.

\paragraph{Series-parallel digraphs.}
A  digraph~$G=(V,E)$ is \emph{series-parallel}\footnote{Series–parallel digraphs come in two versions in the literature, the `edge version' and the `vertex version'~\cite{valdes1982recognition}; here we refer to the more common `edge version'.} if it has a designated start node~$u_0\in V$ and a designated end node~$v_0\in V$ that satisfy one of the following three (recursive) conditions:
\begin{itemize}\itemsep0.3ex
\item $G$ consists of a single arc from~$u_0$ to~$v_0$, that is,~$V=\{u_0,v_0\}$ and~$E=\{(u_0,v_0)\}$.
\item $G$ is a \emph{series composition}, that is, there are two series-parallel digraphs \mbox{$G_1 = (V_1, E_1)$} with designated~$u_1,v_1\in V_1$ and~\mbox{$G_2=(V_2,E_2)$} with designated $u_2,v_2\in V_2$ such that $V=V_1\cup V_2$, $E=E_1\cup E_2$, $u_0=u_1$, $v_1=u_2$, $v_0=v_2$, and~$V_1\cap V_2=\{v_1\}$.
\item $G$ is a \emph{parallel composition}, that is, there are two series-parallel digraphs \mbox{$G_1 = (V_1, E_1)$} with designated~$u_1,v_1\in V_1$ and~\mbox{$G_2=(V_2,E_2)$} with designated~$u_2,v_2\in V_2$ such that $V=V_1\cup V_2$, $E=E_1\cup E_2$, $u_0=u_1=u_2$, $v_0=v_1=v_2$, and~$V_1\cap V_2=\{u_0,v_0\}$.
\end{itemize}

We consider \emph{decompositions} of series-parallel digraphs represented by \emph{$sp$-trees}.
Given a series-parallel digraph~$G=(V,E)$, the corresponding $sp$-tree \mbox{$T=(V_T,E_T)$} is a rooted binary tree where every node~$\spnode\in V_T$ has one of the types \emph{$p$-node}, \emph{$s$-node}, or \emph{$q$-node}, and a label~$(u_{\spnode},v_{\spnode})\in V\times V$.
Then, node~$\spnode$ represents a series-parallel subdigraph~$G_\spnode=(V_\spnode,E_\spnode)$ of~$G$ with start node~$u_{\spnode}$ and end node~$v_{\spnode}$.
The root of the $sp$-tree has label~$(u_0,v_0)$ and represents the entire series-parallel digraph~$G$. The leafs of the $sp$-tree are of type~$q$ and represent the arcs of~$G$. More precisely, a leaf node~$\spnode$ corresponds to the subdigraph~$G_\spnode$ that consists of a single arc from~$u_{\spnode}$ to~$v_{\spnode}$, that is,~$V_\spnode=\{u_\spnode,v_\spnode\}$ and~$E_\spnode=\{(u_{\spnode},v_{\spnode})\}$.
All non-leaf nodes of~$T$ are either of type~$s$ or~$p$. If node~$\spnode\in V_T$ is of type~$p$, then~$G_\spnode$ is obtained by parallel composition of two series-parallel subdigraphs that are represented by the two children~$\spnode_1$ and~$\spnode_2$ of~$\spnode$. If node~$\spnode$ is of type~$s$, then~$G_\spnode$ is a series composition of two series-parallel digraphs represented by the two children~$\spnode_1$ and $\spnode_2$ of~$\spnode$; an illustrative example is given in  \Cref{fig:series-parallel}.
For a node~$\spnode$ of the $sp$-tree, let~$\innerV_{\spnode}$ denote the set of \emph{inner nodes} of subdigraph~$G_{\spnode}$, that is, $\innerV_{\spnode}\coloneqq V_{\spnode}\setminus\{u_{\spnode},v_{\spnode}\}$.

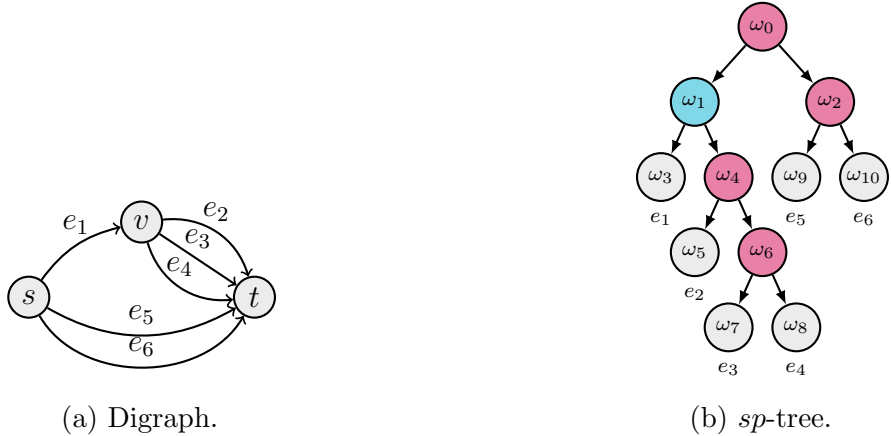
\begin{figure}%
	\begin{subfigure}{.5\textwidth}
		\centering%
		\begin{tikzpicture}
			\node[CircleLight] (s) at (0,0) {$u_0$};
			\node[CircleLight] (v) at (1.5,1) {$v$};
			\node[CircleLight] (t) at (3,0) {$v_0$};

			\draw[->, thick] (s) edge[bend left=20] node[above, midway] {$e_1$} (v);

			\draw[->, thick] (v) edge[bend left=40] node[above, midway] {$e_2$}  (t);
			\draw[->, thick] (v) edge node[above, midway] {$e_3$}  (t);
			\draw[->, thick] (v) edge[bend right=40] node[above, midway] {$e_4$}  (t);

			\draw[->, thick] (s) edge[bend right=30] node[above, midway] {$e_5$}  (t);
			\draw[->, thick] (s) edge[bend right=60] node[above, midway] {$e_6$}  (t);
		\end{tikzpicture}
		\caption{Digraph.}
	\end{subfigure}%
	\begin{subfigure}{.5\textwidth}
		\centering%
		\def\xdist{.45}%
		\def\ydist{1}%
		\begin{tikzpicture}
			\node[PNode] (w0) at (0,0) {\scriptsize$\spnode_0$};

			\node[SNode] (w1) at (-2*\xdist,-\ydist) {\scriptsize$\spnode_1$};
			\node[PNode] (w2) at (2*\xdist,-\ydist) {\scriptsize$\spnode_2$};

			\node[QNode, label = below:{\scriptsize$e_1$}] (w3) at (-3*\xdist,-2*\ydist) {\scriptsize$\spnode_3$};
			\node[PNode] (w4) at (-\xdist,-2*\ydist) {\scriptsize$\spnode_4$};
			\node[QNode, label = below:{\scriptsize$e_5$}] (w9) at (\xdist,-2*\ydist) {\scriptsize$\spnode_9$};
			\node[QNode, label = below:{\scriptsize$e_6$}] (w10) at (3*\xdist,-2*\ydist) {\scriptsize$\spnode_{10}$};

			\node[QNode, label = below:{\scriptsize$e_2$}] (w5) at (-2*\xdist,-3*\ydist) {\scriptsize$\spnode_5$};
			\node[PNode] (w6) at (0,-3*\ydist) {\scriptsize$\spnode_6$};

			\node[QNode, label = below:{\scriptsize$e_3$}] (w7) at (-\xdist,-4*\ydist) {\scriptsize$\spnode_7$};
			\node[QNode, label = below:{\scriptsize$e_4$}] (w8) at (\xdist,-4*\ydist) {\scriptsize$\spnode_8$};

			\draw[-{latex}, thick]
				(w0) edge (w1)
				(w0) edge (w2)
				(w1) edge (w3)
				(w1) edge (w4)
				(w2) edge (w9)
				(w2) edge (w10)
				(w4) edge (w5)
				(w4) edge (w6)
				(w6) edge (w7)
				(w6) edge (w8);
		\end{tikzpicture}
		\caption{$sp$-tree.}
	\end{subfigure}%
	\caption{Series-parallel digraph and its $sp$-tree. Nodes $\spnode_0$, $\spnode_2$, $\spnode_4$, and $\spnode_6$ (red) are classified as type~$p$, while $\spnode_1$ (blue) is classified as type~$s$. The leaves (gray) are of type~$q$ and correspond to the arcs of the digraph.}%
	\label{fig:series-parallel}%
\end{figure}

\paragraph{Multiflows in series-parallel digraphs.}
For a given multiflow~$\vec{\mcflow}$ in~$G$, we are interested in the amount of flow of every commodity that is routed through every subcomponent of the series-parallel graph.
To this end, we use the decomposition of the series-parallel graph given by the $sp$-tree~$T$.
Formally, we define for every~$\spnode \in V_T$ with associated subdigraph~$G_\spnode$ the \emph{demand share $z_{\spnode,i} (\vec{\mcflow}) \in [0,1]$} of commodity~$i\in[k]$ that is routed through the subdigraph $G_\spnode$ by
\[
	z_{\spnode,i} (\vec{\mcflow})\coloneqq
	\begin{cases}
		1 &\text{if~$\innerV_\spnode \cap\{s_i,t_i\}\neq\emptyset$,} \\
		\frac{1}{d_i}\sum_{\substack{e \in \delta^+_{\spnode}(u_{\spnode})}} \mcflow_{e, i} & \text{otherwise.}
	\end{cases}
\]
Here, $\delta^+_{\spnode}(\cdot)$ denotes the set of outgoing arcs from a particular node within~$G_{\spnode}$.
Moreover, let~$\vec{z}^{\spnode} (\vec{\mcflow}) \coloneqq (z_{\spnode, i} (\vec{\mcflow}))_{i \in [k]}$.
This definition is illustrated in \cref{example}.

\begin{definition}
\label{def:routed-frac}
For~$\spnode\in V_T$ and~$i\in[k]$, we say that commodity~$i$ is \emph{routed unsplittably} through the subdigraph~$G_\spnode$ if~$z_{\spnode, i}(\vec{\mcflow}) \in \{0, 1\}$; otherwise, commodity~$i$ is \emph{routed fractionally} through~$G_\spnode$.
Let
$\fractionalDemands_\spnode (\vec{\mcflow}) \coloneqq \{ i \in [k] : 0<z_{\spnode,i} (\vec{\mcflow})<1 \}$
denote the subset of commodities that are routed fractionally through~$G_\spnode$ and let
$\unsplittableDemands_\spnode (\vec{\mcflow}) \coloneqq \{ i \in [k] : z_{\spnode, i} = 1 \}$
denote the subset of commodities that are routed completely through~$G_\spnode$.
\end{definition}

Notice that a multiflow~$\vec{\mcflow}$ is an unsplittable flow if every commodity is routed unsplittably through every subdigraph~$G_{\spnode}$ or, equivalently, if~$\fractionalDemands_{\spnode} (\vec{\mcflow}) = \emptyset$ for every node~$\spnode$ of the $sp$-tree.

\begin{example} \label{example}
In order to illustrate the concepts presented, we use the following example  throughout the paper.
As an instance, we use the series-parallel digraph with corresponding $sp$-tree depicted in \Cref{fig:series-parallel} with the shown $sp$-tree. In this digraph, we consider eight commodities, sharing the same source~$s=u_0$ and sink~$t=v_0$, and with demands $(1,2,1,1,1,2,1,1)$. The demands are satisfied by the multiflow~$\vec{\mcflow}$ given by the following matrix.

\[
\arraycolsep1ex
\begin{array}{c|cccccccc}
\multicolumn{1}{c}{} & \multicolumn{8}{c}{\text{commodity}}\\
 & 1 & 2 & 3 & 4 & 5 & 6 & 7 & 8 \\
\hline
e_1 & 1 & 2 & 1 & 1 & 1 & \nicefrac12 & 0 & 0 \\
e_2 & 1 & \nicefrac54 & 0 & 0 & 0 & 0 & 0 & 0 \\
\tikz[overlay]{\node[rotate=90] at (-1em, 0) {arc};}
e_3 & 0 & \nicefrac34 & 1 & \nicefrac12 & 0 & 0 & 0 & 0 \\
e_4 & 0 & 0 & 0 & \nicefrac12 & 1 & \nicefrac12 & 0 & 0 \\
e_5 & 0 & 0 & 0 & 0 & 0 & \nicefrac32 & \nicefrac12 & 0 \\
e_6 & 0 & 0 & 0 & 0 & 0 & 0 & \nicefrac12 & 1 \\
\end{array}
\]
This multiflow yields the total flow~\[\vec{\totalflow} = (\nicefrac{13}2, \nicefrac94, \nicefrac94, 2, 2, \nicefrac32).\] The flow is feasible for capacities $c_e \coloneqq \totalflow_e$ for all arcs $e$.
Consider the node~$\spnode_6$ of the $sp$-tree. For this node, we have $u_{\spnode_6} = v$ and $\delta^+_{\spnode_6} (v) = \{ e_3, e_4 \}$. For $i = 2$, we get $z_{\spnode,i} (\vec{\mcflow}) = \frac{1}{2} (\nicefrac34 + 0) = \nicefrac{3}{8}$. For all commodities, we have \[\vec{z}^{\spnode_6} (\vec{\mcflow}) = (0, \nicefrac{3}{8}, 1, 1, 1, \nicefrac{1}{4}, 0, 0),\] thus the commodities $2$ and $6$ are routed fractionally through $G_{\spnode_6}$, i.e. $\fractionalDemands_{\spnode_6} (\vec{\mcflow}) = \{2, 6 \}$, and the commodities $3, 4$, and $5$ are routed unsplittably through~$G_{\spnode_6}$.
\end{example}

We outline three fundamental properties of the relationship between the demand shares in the subdigraph of a non-leaf node and the subdigraphs corresponding to its child nodes.

\begin{observation}
\label{obs:z-in-series-parallel}
Let~$\spnode\in V_T$ be a non-leaf node with child nodes~$\spnode_1,\spnode_2$ and let~$i\in[k]$ be some commodity.
\begin{enumerate}[(a)]
\item\label{obs:z-in-parallel}
If $\spnode$ is a $p$-node, then $z_{\spnode,i} (\vec{\mcflow}) = z_{\spnode_1,i} (\vec{\mcflow}) + z_{\spnode_2,i} (\vec{\mcflow})$.
\item\label{obs:z-in-serial}
If $\spnode$ is an $s$-node and $s_i,t_i \notin \innerV_{\spnode}$, then $z_{\spnode,i} (\vec{\mcflow}) = z_{\spnode_1,i} (\vec{\mcflow}) = z_{\spnode_2,i} (\vec{\mcflow})$.
\item\label{obs:z-in-serial-part2}
If $\spnode$ is an $s$-node and $\{s_i,t_i\}\cap\innerV_{\spnode}\neq\emptyset$, then~$z_{\spnode,i}(\vec{\mcflow})=1$ and $z_{\spnode_1,i}(\vec{\mcflow}),z_{\spnode_2,i}(\vec{\mcflow})\in\{0,1\}$.
\end{enumerate}
\end{observation}

\section{The cut condition and integer multiflows}
\label{sec:cut-condition}

The following \emph{cut condition} is a necessary criterion for the existence of a feasible (fractional) multiflow (see, e.g.,~\cite[Chapter~70.7]{Schrijver03}):
\begin{align}
\label{eq:cut-condition}
c\bigl(\delta^+(X)\bigr)\geq\sum_{i\in[k]:s_i\in X\subseteq V\setminus\{t_i\}}d_i\qquad\text{for all~$X\subseteq V$.}
\end{align}
Here,~$\delta^+(X)$ denotes the set of arcs with tail in~$X$ and head in the complement~$V\setminus X$. To emphasize the underlying digraph~$G$, we sometimes write~$\delta^+_G(X)$ instead of~$\delta^+(X)$.
However, in series-parallel digraphs, the cut condition generally does not guarantee the existence of a feasible multiflow. The instance in \Cref{fig:cut-cond} is a counterexample, in which the cut condition~\eqref{eq:cut-condition} is satisfied, yet no feasible multiflow exists.
Note that the only~$s_1$--$t_1$~path in the given network passes through node~$s_2$, meaning that all flow of commodity~$1$ must be routed through~$s_2$.
We thus obtain an equivalent multiflow instance by subdividing commodity~$1$ into two subcommodities~$1_1$ and~$1_2$ with~$s_{1_1}=s_1$, $t_{1_1}=s_2=s_{1_2}$, and~$t_{1_2}=t_1$, where both subcommodities inherit the unit demand value of the original commodity, i.e.,~$d_{1_1}=d_{1_2}=d_1=1$. The resulting multiflow instance violates the cut condition since~$c(\delta^+(\{s_2\}))=1$ but commodities~$1_2$ and~$2$ both need to send one unit of flow across cut~$\delta^+(\{s_2\})$.
In the context of multiflows in undirected graphs, the trick of subdividing commodities was used, for example, in~\cite{chekuri2013flow,OkamuraSeymour1981}.

\subsubsection*{Aligned multiflow instances on series-parallel digraphs.}

We generalize the simple idea of subdividing commodities outlined above.
If, for some commodity~$i\in[k]$, there is a node~$v\in V\setminus\{s_i,t_i\}$ that is on every~$s_i$--$t_i$~path in~$G$, commodity~$i$ can be subdivided into two commodities~$i_1$ and~$i_2$ with~$s_{i_1}=s_i$, $t_{i_1}=v=s_{i_2}$, $t_{i_2}=t_i$, and~$d_{i_1}=d_{i_2}=d_i$. There is a straightforward one-to-one correspondence between multiflows for the original instance and multiflows for the subdivided instance. In particular, unsplittable multiflows for the original instance are mapped to unsplittable multiflows for the subdivided instance, and vice versa.
%
If we iteratively subdivide commodities as far as possible, we arrive at an equivalent multiflow instance on~$G$ that features a particular structure described in the following definition.

\begin{definition}
\label{def:aligned}
The commodities of a multiflow instance on a series-parallel digraph~$G$ with $sp$-tree~$T$ are \emph{aligned} with the series-parallel decomposition if, for every commodity~$i\in[k]$, there is a component~$G_\spnode$, $\spnode\in V_T$, with~$s_i=u_\spnode$ and~$t_i=v_\spnode$.
In this case, we also say that the multiflow instance is \emph{aligned}.
\end{definition}

Note that the choice of~$\spnode\in V_T$ for commodity~$i\in[k]$ in \Cref{def:aligned} is not necessarily unique. If, for example,~$\spnode$ is a~$p$-node, then its children have the same start and end node, and are thus alternative choices.
Whenever we refer to \emph{the node~$\spnode$ corresponding to commodity~$i$} in an aligned multiflow instance, we mean the unique node~$\spnode_i\in V_T$  with~$s_i=u_{\spnode_i}$ and~$t_i=v_{\spnode_i}$ for which all~$s_i$--$t_i$~paths in~$G$ lie in~$G_{\spnode_i}$ or, equivalently, which has minimum distance to the root node of~$T$.
Note that commodity~$i$ must then be routed entirely within~$G_{\spnode_i}$.

Any multiflow instance on a series-parallel digraph can be efficiently reduced to an aligned instance by repeatedly subdividing commodities.

\begin{observation}
\label{obs:aligned}
For a non-aligned multiflow instance on a series-parallel digraph, there is an~$i\in[k]$ and~$v\in V\setminus\{s_i,t_i\}$ such that~$v$ is on every~$s_i$--$t_i$~path.
\end{observation}

\begin{proof}
Consider a commodity~$i\in[k]$ that violates the condition in Definition~\ref{def:aligned}.
For some arc~$e\in\delta^+(s_i)$ and the corresponding leaf node~$\spnode\in V_T$ (i.e.,~$E_\spnode=\{e\}$), it holds that~$s_i=u_\spnode$. Now choose a node~$\spnode\in V_T$ with~$s_i=u_\spnode$ that has minimum distance from the root node of~$T$. If~$t_i\notin V_\spnode$, then we can set~$v\coloneqq v_\spnode$ since, by our choice of~$\spnode$, the end node~$v_\spnode$ is on every path that starts at~$s_i$ and leaves~$G_\spnode$. Otherwise,~$t_i\in\innerV_{\spnode}$ and we choose a node~$\spnode'\in V_T$ with~$v_{\spnode'}=t_i$ that has minimum distance from the root node of~$T$. Since~$t_i\in\innerV_{\spnode}$, node~$\spnode'$ must be a descendant of~$\spnode$ in~$T$. Moreover,~$s_i\neq u_{\spnode'}$ by choice of~$i$. Thus,~$s_i\notin V_{\spnode'}$, and we can set~$v\coloneqq u_{\spnode'}$ since, by our choice of~$\spnode'$, the start node~$u_{\spnode'}$ is on every path that starts outside of~$G_{\spnode'}$ and ends at~$t_i$.
\end{proof}

\subsubsection*{Solving aligned instances via single-commodity flows.}

A relaxation of a given multiflow instance can be obtained by allowing flow to be sent from any source~$s_i$ to any sink~$t_j$ where~$j$ is not necessarily equal to~$i$. This relaxation yields a single-commodity $b$-transshipment instance where
\[
b_v\coloneqq \sum_{i\in[k]:s_i=v}d_i-\sum_{i\in[k]:t_i=v}d_i.
\]
Then, for any multiflow~$\vec{\mcflow}$, its total flow~$\vec{\totalflow}$ is a~$b$-transshipment. In particular, for a feasible multiflow instance, the corresponding~$b$-transshipment instance is feasible as well. The opposite direction, however, is not true. Here is a simple counterexample on a unit-capacity series-parallel digraph:
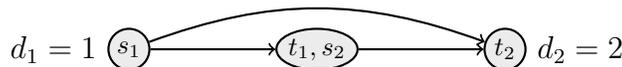
\begin{center}
\begin{tikzpicture}
\node [CircleLight, label = left:{$d_1=1$}] (s1) at (-2.5,0) {\footnotesize$s_1$};
\node [draw, thick, ellipse, inner sep=0pt, minimum width = 5.5ex, minimum height=3ex,fill=gray!15] (t1) at (0,0) {\footnotesize$t_1,s_2$};
\node [CircleLight, label = right:{$d_2=2$}] (t2) at (2.5,0) {\footnotesize$t_2$};
\draw [->,thick] (s1) to (t1);
\draw [->,thick] (s1) to [out=20,in=160] (t2);
\draw [->,thick] (t1) to (t2);
\end{tikzpicture}
\end{center}
Note that any multiflow must send the entire demand~$d_2=2$ across the unit-capacity arc~$(s_2,t_2)$, and is thus infeasible. In the relaxation, however, the supply of value~$2$ at~$s_2$ together with the unit demand at the same node~$t_1=s_2$ results in a total supply of~$1$ that can be feasibly sent across arc~$(s_2,t_2)$, while~$s_1$ sends its unit supply via arc~$(s_1,t_2)$ to~$t_2$. What makes the~$b$-transshipment instance feasible is thus the canceling of supply and demand at node~$t_1=s_2$. But if we split this node into two nodes that are connected by an infinite capacity arc, we arrive at an equivalent (infeasible) multiflow instance on a series-parallel digraph, whose corresponding $b$-transshipment instance is infeasible as well:
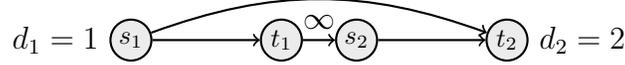
\begin{center}
\begin{tikzpicture}
\node [CircleLight, label = left:{$d_1=1$}] (s1) at (-2.5,0) {\footnotesize$s_1$};
\node [CircleLight] (t1) at (-0.5,0) {\footnotesize$t_1$};
\node [CircleLight] (s2) at (0.5,0) {\footnotesize$s_2$};
\node [CircleLight, label = right:{$d_2=2$}] (t2) at (2.5,0) {\footnotesize$t_2$};
\draw [->,thick] (s1) to (t1);
\draw [->,thick] (s1) to [out=20,in=160] (t2);
\draw [->,thick] (t1) to node [above] {$\infty$} (s2);
\draw [->,thick] (s2) to (t2);
\end{tikzpicture}
\end{center}
The simple idea of splitting a node~$v$ that is both a source and a sink works in general. Given an aligned multiflow instance on a series-parallel digraph with such a node~$v$, we obtain an equivalent aligned multiflow instance on a modified series-parallel digraph as follows: split node~$v$ into two nodes~$v_1$ and~$v_2$ joined by an infinite capacity arc~$(v_1,v_2)$; the head node of every arc~$e\in\delta^-(v)$ is set to~$v_1$, and the tail node of every arc~$e\in\delta^+(v)$ is set to~$v_2$; moreover, for~$i\in[k]$, if~$t_i=v$, we set~$t_i\coloneqq v_1$, and if~$s_i=v$, we set~$s_i\coloneqq v_2$. Note that the resulting multiflow instance is indeed equivalent to the original one, the modified digraph is still series-parallel, and the new multiflow instance is still aligned.

\begin{lemma}
\label{lem:transshipment}
For an aligned multiflow instance on a series-parallel digraph with $s_i\neq t_j$ for all~$i\neq j$, every $b$-transshipment~$\vec{y}=(y_e)_{e\in E}$ can be efficiently turned into a multiflow~$\vec{\mcflow}$ whose total flow~$\vec{\totalflow}$ equals~$\vec{y}$.
\end{lemma}

\Cref{lem:transshipment} implies that aligned multiflow instances with $s_i\neq t_j$ for all~$i\neq j$ inherit the sufficiency of the cut condition as well as strong integrality properties from single-commodity flows. Moreover, as we argue above, any multiflow instance on a series-parallel digraph can be reduced to such an aligned instance in a straightforward manner. This concludes the proof of \Cref{thm:integrality}.

\begin{proof}[Proof of \Cref{lem:transshipment}]
We prove the existence of multiflow~$\vec{\mcflow}$ by induction on the number of commodities~$k$. The proof is constructive and yields an efficient algorithm for finding~$\vec{\mcflow}$. The case~$k=1$ is trivial since~$\vec{y}$ is an~$s_1$--$t_1$~flow of value~$d_1$, and thus also a multiflow.

If~$k>2$, let~$i\in[k]$ be a commodity such that the corresponding node~$\spnode\coloneqq\spnode_i$ in~$T$ has maximum distance from the root of~$T$. Then, there is no commodity~$j\in[k]$ such that~$\spnode_j$ is a descendant of~$\spnode$. Therefore,~$\innerV_{\spnode}$ does not contain a source or sink of any commodity. In particular, the restriction of~$\vec{y}$ to~$G_{\spnode}$ is an~$u_{\spnode}$--$v_{\spnode}$~flow. Since by assumption~$s_i=u_{\spnode}\neq t_j$ for all~$j\in[k]$, the value of this~$u_{\spnode}$--$v_{\spnode}$~flow is at least~$d_i$. We can therefore find an~$u_{\spnode}$--$v_{\spnode}$~flow~$\vec{z}$ in~$G_{\spnode}$ of value~$d_i$ such that~$z_e\leq y_e$ for all~$e\in E_{\spnode}$. Then,~$\vec{y'}\coloneqq\vec{y}-\vec{z}$ is a~$b'$-transshipment corresponding to the reduced multiflow instance that is obtained by deleting commodity~$i$. By induction, there is a multiflow~$\vec{\mcflow'}$ for this reduced instance whose total flow~$\vec{\totalflow'}$ equals~$\vec{y'}$. Complementing~$\vec{\mcflow'}$ with~$\vec{z}$ as the flow of commodity~$i$ yields the desired multiflow~$\vec{\mcflow}$ with total flow~$\vec{\totalflow}=\vec{y}$.
\end{proof}

\subsubsection*{Sufficiency of a strengthened cut condition}

As discussed above, the classical cut condition~\eqref{eq:cut-condition} is, in general, not sufficient to guarantee the existence of a feasible multiflow in a series-parallel digraph~$G=(V,E)$; see the example in \Cref{fig:cut-cond}. A natural strengthening of the classical cut condition considers, for any subset of arcs~$F\subseteq E$, the set of commodities
\[
K(F)\coloneqq\{i\in[k]\mid\text{there is no~$s_i$--$t_i$~path in the subdigraph~$(V,E\setminus F)$}\}.
\]
By definition, the set~$F$ forms a \emph{multicut} for the commodities in~$K(F)$; that is, none of these commodities can be routed after removing the arcs in~$F$. In any feasible multiflow, the arcs in~$F$ must therefore carry at least the total demand of all commodities in~$K(F)$. This yields the following \emph{strong cut condition}:
\begin{align}
\label{eq:strong-cut-condition}
c(F)\geq\sum_{i\in K(F)}d_i\qquad\text{for all~$F\subseteq E$.}
\end{align}
This strong cut condition is clearly necessary for the existence of a feasible multiflow.

In the example shown in \Cref{fig:cut-cond}, the classical cut condition~\eqref{eq:cut-condition} is satisfied, but the strong cut condition~\eqref{eq:strong-cut-condition} is violated for the arc set~$F=\{(s_2,t_1)\}$: here,~$c(F)=1$, while the total demand of the blocked commodities~$K(F)=\{1,2\}$ equals~$d_1+d_2=2$. Notably, in this example the violating arc set~$F$ is of the form~$\delta^+(X)$ for~$X=\{s_2\}\subseteq V$.

In fact, as we show in \Cref{thm:strengthened-cut-condition}, for any infeasible multiflow instance on a series-parallel digraph, there always exists a node subset~$X\subseteq V$ such that the strong cut condition is violated for~$F=\delta^+(X)$. This observation motivates a further refinement: the following \emph{strengthened cut condition}
\begin{align}
\label{eq:strengthened-cut-condition}
c\bigl(\delta^+(X)\bigr)\geq\sum_{i\in K(\delta^+(X))}d_i\qquad\text{for all~$X\subseteq V$,}
\end{align}
is stronger than the classical cut condition~\eqref{eq:cut-condition} but weaker than the strong cut condition~\eqref{eq:strong-cut-condition}. In \Cref{fig:strengthened-vs-strong} we provide an example of an infeasible multiflow instance on an acyclic digraph that is not series-parallel, which satisfies the strengthened cut condition~\eqref{eq:strengthened-cut-condition} but violates the strong cut condition~\eqref{eq:strong-cut-condition}.
\begin{figure}[t]
\centering
\begin{tikzpicture}
\node [CircleLight, label = left:{\footnotesize$d_1=2$}] (s1) at (-3,0) {\footnotesize$s_1$};
\node [CircleLight] (m) at (0,0) {};
\node [CircleLight, label = right:{\footnotesize$d_2=2$}] (t2) at (3,0) {\footnotesize$t_2$};

\node [CircleLight] (e1) at (-1,-1.5) {};
\node [CircleLight] (e2) at (1,-1.5) {};

\node [CircleLight] (s2) at (-2,-0.75) {\footnotesize$s_2$};
\node [CircleLight] (t1) at (2,-0.75) {\footnotesize$t_1$};

\draw [->,thick] (s1) to (m);
\draw [->,thick] (m) to (t2);
\draw [->,thick] (e1) to (e2);
\draw [->,thick,dashed] (s1) to [out=-90,in=180] (e1);
\draw [->,thick,dashed] (e2) to [out=0,in=-90] (t2);
\draw [->,thick,dashed] (s2) to (e1);
\draw [->,thick,dashed] (e2) to (t1);
\draw [->,thick,dashed] (s2) to (m);
\draw [->,thick,dashed] (m) to (t1);
\end{tikzpicture}
\caption{An infeasible multiflow instance on an acyclic digraph. Solid arcs have unit capacity; dashed arcs have infinite capacity. The strong cut condition~\eqref{eq:strong-cut-condition} is violated only for the arc set~$F=\{\text{solid arcs}\}$. However, since this arc set is not of the form~$\delta^+(X)$ for any~$X\subseteq V$, the strengthened cut condition~\eqref{eq:strengthened-cut-condition} remains satisfied.}
\label{fig:strengthened-vs-strong}
\end{figure}

\begin{theorem}\label{thm:strengthened-cut-condition}
For multiflow instances on series-parallel digraphs, the strengthened cut condition~\eqref{eq:strengthened-cut-condition} is necessary and sufficient for the existence of a feasible multiflow.
\end{theorem}

\begin{proof}
The strengthened cut condition~\eqref{eq:strengthened-cut-condition} is obviously necessary for the existence of a feasible multiflow. It remains to show that for every infeasible multiflow instance on a series-parallel digraph~$G=(V,E)$, at least one of the inequalities in~\eqref{eq:strengthened-cut-condition} is violated. We refer to the relaxation obtained by eliminating all commodities with source~$u_0$ (start node of~$G$) and sink~$v_0$ (end node of~$G$) as the \emph{reduced instance on~$G$}. To establish \Cref{thm:strengthened-cut-condition}, we in fact prove a stronger statement that directly implies it.

\begin{claim}
For every infeasible multiflow instance on a series-parallel digraph~$G=(V,E)$, there exists a subset~$X\subseteq V$ such that the corresponding inequality in~\eqref{eq:strengthened-cut-condition} is violated. Moreover, the violating set~$X$ can be chosen such that the following property holds:~$X \cap \{u_0,v_0\} = \{u_0\}$ if and only if the reduced instance on~$G$ is feasible.
\end{claim}

We prove the claim by induction on the number of arcs in~$G$. For the base case~$|E|=1$, the digraph consists of a single arc~$(u_0,v_0)$ whose capacity is insufficient to support the total demand of all commodities from~$u_0$ to~$v_0$. Hence, the strengthened cut condition~\eqref{eq:strengthened-cut-condition} is violated for~$X=\{u_0\}$. Note that, in this case, the \emph{reduced} instance on~$G$ is feasible.

If~$|E|\geq 2$, then, as described in \Cref{sec:preliminaries}, the digraph~$G$ is either a series or a parallel composition of two series-parallel digraphs~$G_1=(V_1,E_1)$ and~$G_2=(V_2,E_2)$, each with strictly fewer arcs than~$G$. As in \Cref{sec:preliminaries}, for~$j=1,2$, we denote the start and end nodes of~$G_j$ by~$u_j$ and~$v_j$, respectively.

\paragraph{Case~1.} If~$G$ is a \emph{series} composition of~$G_1$ and~$G_2$ (note that~$v_1=u_2$ in this case), then the multiflow instance on~$G$ naturally decomposes into two independent subinstances, one on~$G_1$ and one on~$G_2$, as follows: Every commodity whose source and sink both lie in~$V_1$ (respectively, in~$V_2$) is assigned to the subinstance on~$G_1$ (respectively, on~$G_2$). All remaining commodities have their source in~$V_1\setminus\{v_1\}$ and their sink in~$V_2\setminus\{u_2\}$. These commodities are split across the two subinstances: each is assigned to the subinstance on~$G_1$ with modified sink node~$v_1$, and to the subinstance on~$G_2$ with modified source node~$u_2$. Note that a feasible multiflow on~$G$ corresponds to a pair of feasible multiflows on~$G_1$ and~$G_2$. In particular, the original multiflow instance on~$G$ is feasible if and only if both subinstances are feasible. We may assume that the subinstance on~$G_1$ is infeasible; the case where~$G_2$ is infeasible is symmetric.

By the induction hypothesis, there exists a subset~$X \subseteq V_1$ such that the corresponding inequality in~\eqref{eq:strengthened-cut-condition} is violated for the subinstance on~$G_1$. If the reduced instance on~$G$ is feasible, then also the reduced subinstance on~$G_1$ must be feasible and, by induction,~$X$ can be chosen such that~$X \cap \{u_1,v_1\}=\{u_1\}$. Therefore~$\delta^+_{G_1}(X)=\delta^+_G(X)$ and this set of arcs separates a commodity in~$G_1$ if and only if it separates the corresponding commodity in~$G$. Hence, the set~$X$ also yields a violated inequality for the original instance on~$G$, and~$X\cap\{u_0,v_0\}=\{u_0\}$.

If, on the other hand, already the reduced instance on~$G$ is infeasible, we consider the set~$X\cup (V_2 \setminus \{u_2\})$. Observe that
\[
\delta^+_{G_1}(X) = \delta^+_G\bigl(X \cup (V_2 \setminus \{u_2\})\bigr).
\]
Again, this set of arcs separates a commodity in~$G_1$ if and only if it separates the corresponding commodity in~$G$. Therefore, the inequality in~\eqref{eq:strengthened-cut-condition} associated with~$X \cup (V_2 \setminus \{u_2\})$ is violated in the original instance. Moreover, also the additional property stated in the claim is satisfied since~$v_0=v_2\in X \cup (V_2 \setminus \{u_2\})$.

\paragraph{Case~2.} 
If~$G$ is a \emph{parallel} composition of~$G_1$ and~$G_2$, then~$u_0=u_1=u_2$ and~$v_0=v_1=v_2$. We first consider the situation in which already the reduced instance on~$G$ is infeasible. The reduced instance on~$G$ naturally decomposes into two independent (reduced) subinstances on~$G_1$ and~$G_2$, since each commodity must be routed entirely within either~$G_1$ or~$G_2$. Consequently, if the reduced instance on~$G$ is infeasible, then at least one of the two subinstances must be infeasible. Without loss of generality, assume that infeasibility arises in~$G_1$. By the induction hypothesis, there exists a subset~$X\subseteq V_1$ with \mbox{$X\cap\{u_1,v_1\}\neq\{u_1\}$} such that the corresponding inequality in~\eqref{eq:strengthened-cut-condition} is violated for the subinstance on~$G_1$. If~\mbox{$u_1\notin X$}, then~$\delta^+_{G_1}(X)=\delta^+_G(X)$ and, also for the original instance on~$G$, the inequality in~\eqref{eq:strengthened-cut-condition} corresponding to~$X$ is violated. Otherwise, if~$u_1\in X$, then also~$v_1\in X$ and~$\delta^+_{G_1}(X)=\delta^+_G(X\cup V_2)$ such that the inequality corresponding to~$X\cup V_2$ in~\eqref{eq:strengthened-cut-condition} is violated for the instance on~$G$. Moreover, also~$X\cup V_2$ contains both~$u_0=u_1$ and~$v_0=v_1$ such that the additional property stated in the claim is fulfilled.

It remains to consider the situation where the reduced instance on~$G$ is feasible but the full instance is infeasible. In particular, there are commodities with source~$u_0$ and sink~$v_0$ and a total demand of~$D>0$. Without loss of generality, we combine these commodities into a single commodity, denoted as commodity~$0$, with demand~$d_0\coloneqq D$. For~$j=1,2$, let~$\delta_j\geq0$ denote the maximum amount of flow that a feasible multiflow can send from~$u_1=u_0$ to~$v_j=v_0$ in addition to the commodities of the (reduced) subinstance on~$G_j$. Observe that~$\delta_1+\delta_2<d_0$ since otherwise there is a feasible multiflow for the full instance on~$G$. Let~$\varepsilon\coloneqq (d_0-\delta_1-\delta_2)/2>0$. Then, for~$j=1,2$, adding an $u_j$--$v_j$~commodity of demand~$\delta_j+\varepsilon$ to the (reduced) subinstance on~$G_j$ results in an infeasible instance on~$G_j$. By the induction hypothesis, there are subsets~$X_1\subseteq V_1$ and~$X_2\subseteq V_2$, both containing~$u_0$ but not~$v_0$, such that the corresponding inequalities~\eqref{eq:strengthened-cut-condition} are violated for these instances on~$G_1$ and~$G_2$, respectively. We set~$X\coloneqq X_1\cup X_2$ and observe that~$\delta^+_{G}(X)$ is the disjoint union of~$\delta^+_{G_1}(X_1)$ and~$\delta^+_{G_2}(X_2)$. Therefore,
\begin{align*}
c\bigl(\delta^+_{G}(X)\bigr)
&=c\bigl(\delta^+_{G_1}(X_1)\bigr)+c\bigl(\delta^+_{G_2}(X_2)\bigr)\\
&<(\delta_1+\varepsilon)+\sum_{i\in K(\delta^+_{G_1}(X_1))}d_i + (\delta_2+\varepsilon) + \sum_{i\in K(\delta^+_{G_2}(X_2))}d_i = \sum_{i\in K(\delta^+_G(X))}d_i,
\end{align*}
where the last equation follows since~$K\bigl(\delta^+_G(X)\bigr)$ is the disjoint union of the sets of commodities~$K\bigl(\delta^+_{G_1}(X_1)\bigr)$, $K\bigl(\delta^+_{G_2}(X_2)\bigr)$, and~$\{0\}$ with~$d_0=(\delta_1+\varepsilon)+(\delta_2+\varepsilon)$.
Thus, the inequality corresponding to~$X$ in~\eqref{eq:strengthened-cut-condition} is violated for the multiflow instance on~$G$. Also the additional property stated in the claim is fulfilled since~$X$ contains~$u_0$ but not~$v_0$.

This concludes the proof of the claim and thus of \Cref{thm:strengthened-cut-condition}.
\end{proof}

As a consequence of \Cref{thm:strengthened-cut-condition}, for series-parallel digraphs, the strengthened cut condition~\eqref{eq:strengthened-cut-condition} implies the strong cut condition~\eqref{eq:strong-cut-condition} and both can be efficiently separated. In this context it is interesting to mention that finding a subset of arcs~$F\subseteq E$ of minimum total capacity~$c(F)$ with~$K(F)=[k]$ is known to be NP-hard~\cite{GargVY04}. Majthoub Almoghrabi recently proved that the problem is NP-hard even on series-parallel digraphs~\cite{Diss-Mohammed}.

\begin{remark}
It may be tempting to think that \Cref{thm:strengthened-cut-condition} follows easily from \Cref{thm:integrality}\,\eqref{thm:integrality:b} or, more precisely, from \Cref{lem:transshipment}. Indeed, for the special case of \emph{aligned} multiflow instances, the classical cut condition~\eqref{eq:cut-condition} already suffices to guarantee the existence of a feasible multiflow, as briefly discussed after \Cref{lem:transshipment}. Although any multiflow instance can be reduced to an aligned one, this reduction introduces complications regarding all three cut conditions~\eqref{eq:cut-condition},~\eqref{eq:strong-cut-condition}, and~\eqref{eq:strengthened-cut-condition}, due to the subdivision of commodities. After subdividing a commodity, its demand may contribute multiple times to the right-hand side of these inequalities, a phenomenon that the cut conditions do not account for in the original, non-aligned instance.
\end{remark}

It is not difficult to observe that the Nagamochi--Ibaraki Theorem~\cite{NagamochiIbaraki1989} (see
also~\cite[Theorem~70.8]{Schrijver03}), also holds if we replace the classical cut condition~\eqref{eq:cut-condition} with the
strengthened cut condition~\eqref{eq:strengthened-cut-condition} or with the strong
cut condition~\eqref{eq:strong-cut-condition}. That is, if, for a fixed digraph and any integer arc capacities and demands, the
\emph{strengthened} cut condition~\eqref{eq:strengthened-cut-condition} or the \emph{strong} cut condition~\eqref{eq:strong-cut-condition} guarantees the existence of a feasible multiflow, then it also guarantees the
existence of a feasible integer multiflow. In view of \Cref{thm:strengthened-cut-condition}, this observation sheds new light on our integrality result in \Cref{thm:integrality}\,\eqref{thm:integrality:a}.


\section{Almost unsplittable flows}
\label{sec:almost-unsplittable}

Before we prove out main result \cref{thm:main}, we first observe that we can turn any multiflow~$\vec{X}$ into an \emph{almost unsplittable} multiflow with the same demand. Formally, we say that a multiflow~$\vec{\mcflow}$ in~$G$ is \emph{almost unsplittable} if at most two commodities are routed fractionally through every subdigraph~$G_{\spnode}$, $\spnode\in V_T$, and the two components in a parallel composition may share at most one fractional commodity.

\begin{definition}
\label{def:almost-unsplittable}
A multiflow~$\vec{\mcflow}$ in~$G$ is \emph{almost unsplittable} if it satisfies:
\begin{enumerate}[(i)]\itemsep0.4ex
\item\label{def:almost-unsplittable:i}
for every node~$\spnode\in V_T$ of the $sp$-tree~$T$, we have $|\fractionalDemands_{\spnode} (\vec{\mcflow})| \leq 2$;
\item\label{def:almost-unsplittable:ii}
$|\fractionalDemands_{\spnode_1}(\vec{\mcflow})\cap\fractionalDemands_{\spnode_2}(\vec{\mcflow})|\leq 1$ for every $p$-node with two children~$\spnode_1,\spnode_2\in V_T$.
If~$i \in \fractionalDemands_{\spnode_1}(\vec{\mcflow})\cap\fractionalDemands_{\spnode_2}(\vec{\mcflow})$, we say that commodity~$i$ is \emph{split between~$\spnode_1$ and~$\spnode_2$}.
\end{enumerate}
\end{definition}

It turns out that any multiflow can be turned into an almost unsplittable multiflow without changing its total flow.

\begin{theorem}
\label{thm:almost-unsplittable}
Any multiflow in~$G$ can efficiently be turned into an almost unsplittable multiflow that satisfies the same demands and has the same total flow.
\end{theorem}
The proof of \Cref{thm:almost-unsplittable} relies on a simple swapping operation that, for given children of a common $p$-node, exchanges flow of two commodities between the corresponding parallel components of~$G$, and is described in the following lemma.

\begin{lemma}
\label{lem:swap}
Consider a multiflow~$\vec{\mcflow}$, a $p$-node~$\bar\spnode\in V_T$ with child nodes~$\spnode_1,\spnode_2$, and two distinct commodities~$i_1\in\fractionalDemands_{\spnode_1}(\vec{\mcflow})$ and~$i_2\in\fractionalDemands_{\spnode_2}(\vec{\mcflow})$. Commodities~$i_1,i_2$ can be rerouted within~$G_{\bar\spnode}$ to obtain a multiflow~$\tilde{\vec{\mcflow}}$ with the same total flow, such that~$z_{\spnode_1,i_1}(\tilde{\vec{\mcflow}})=0$ or~$z_{\spnode_2,i_2}(\tilde{\vec{\mcflow}})=0$.
\end{lemma}

\begin{proof}
While commodities~$i_1$ and~$i_2$ are routed fractionally through the subgraphs~$G_{\spnode_1}$ and~$G_{\spnode_2}$, respectively, we repeatedly swap flow as follows:
\begin{quote}
Find a flow-carrying $u_{\spnode}$--$v_{\spnode}$~path~$P_1$ for commodity~$i_1$ in~$G_{\spnode_1}$, and a flow-carrying $u_{\spnode}$--$v_{\spnode}$~path~$P_2$ for commodity~$i_2$ in~$G_{\spnode_2}$; that is,~$\mcflow_{e,i_j}>0$ for~$e\in E(P_j)$ and~\mbox{$j=1,2$}. Let~$\delta\coloneqq\min_{j\in\{1,2\}}\min_{e\in P_j}\mcflow_{e,i_j}$. Then, for~$j=1,2$, decrease flow of commodity~$i_j$ along path~$P_j$ by~$\delta$ and increase it by~$\delta$ along the other path~$P_{3-j}$.
\end{quote}
The described swapping of flow within~$G_{\bar\spnode}$ maintains a multiflow without changing the total flow. Moreover, the number of flow-carrying arcs of commodity~$i_1$ in~$G_{\spnode_1}$ plus the number of flow-carrying arcs of commodity~$i_2$ in~$G_{\spnode_2}$ is decreased by at least one (in fact, the bottleneck arc~$e$ in the definition of~$\delta$ is no longer counted). Therefore, after at most~$|E_{\bar\spnode}|$ many iterations, commodity~$i_1$ is no longer routed through~$G_{\spnode_1}$ or commodity~$i_2$ is no longer routed through~$G_{\spnode_2}$.
\end{proof}

The following corollary states an important special case of Lemma~\ref{lem:swap}.

\begin{corollary}
\label{cor:swap}
Consider a multiflow~$\vec{\mcflow}$, a $p$-node $\bar\spnode\in V_T$ with child nodes~$\spnode_1,\spnode_2$, and two distinct commodities~$i_1,i_2\in\fractionalDemands_{\spnode_1}(\vec{\mcflow})\cap\fractionalDemands_{\spnode_2}(\vec{\mcflow})$. Commodities~$i_1,i_2$ can be rerouted within~$G_{\bar\spnode}$ to obtain a multiflow~$\tilde{\vec{\mcflow}}$ with the same total flow, such that~$i_1\notin\fractionalDemands_{\spnode_1}(\tilde{\vec{\mcflow}})\cap\fractionalDemands_{\spnode_2}(\tilde{\vec{\mcflow}})$ or~$i_2\notin\fractionalDemands_{\spnode_1}(\tilde{\vec{\mcflow}})\cap\fractionalDemands_{\spnode_2}(\tilde{\vec{\mcflow}})$; in particular,~$|\fractionalDemands_{\spnode_1}(\tilde{\vec{\mcflow}})\cap\fractionalDemands_{\spnode_2}(\tilde{\vec{\mcflow}})|<|\fractionalDemands_{\spnode_1}(\vec{\mcflow})\cap\fractionalDemands_{\spnode_2}(\vec{\mcflow})|$.
\end{corollary}

Notice that one can establish Property~\eqref{def:almost-unsplittable:ii} in Definition~\ref{def:almost-unsplittable} by repeatedly applying Corollary~\ref{cor:swap} until at most one commodity remains that is fractionally routed through both parallel components~$G_{\spnode_1}$ and~$G_{\spnode_2}$.

\begin{proof}[Proof of Theorem~\ref{thm:almost-unsplittable}]
The idea of the proof is to go through the $sp$-tree~$T$ from top to bottom (i.e., from the root node to the leafs), always modifying the current multiflow with the help of Lemma~\ref{lem:swap} until the properties in Definition~\ref{def:almost-unsplittable} are fulfilled. For each node of the $sp$-tree, the swapping routine described in the proof of Lemma~\ref{lem:swap} will be applied polynomially often, such that the entire procedure is efficient.

For the root node~$\spnode_0$ of~$T$, by definition, no commodity is routed fractionally through~$G_{\spnode_0}$. That is,~$\fractionalDemands_{\spnode_0}(\vec{\mcflow})=\emptyset$ and Property~\eqref{def:almost-unsplittable:i} in Definition~\ref{def:almost-unsplittable} is thus fulfilled for~$\spnode_0$. Moreover, since~$\spnode_0$ does not have a parent node in~$T$, Property~\eqref{def:almost-unsplittable:ii} trivially holds.

We next consider an arbitrary node~$\spnode\in V_T$ whose parent node~$\bar{\spnode}$ fulfills Property~\eqref{def:almost-unsplittable:i}, i.e.,~$|\fractionalDemands_{\bar\spnode} (\vec{\mcflow})| \leq 2$. If~$\bar\spnode$ is an $s$-node, then, by Observation~\ref{obs:z-in-series-parallel}\,\eqref{obs:z-in-serial} and~\eqref{obs:z-in-serial-part2}, Property~\eqref{def:almost-unsplittable:i} is also fulfilled for~$\spnode$, and Property~\eqref{def:almost-unsplittable:ii} trivially holds.

It remains to consider the case that~$\bar{\spnode}$ is a $p$-node. Let~$\spnode_1\coloneqq\spnode$ and denote~$\bar{\spnode}$'s second child by~$\spnode_2$. With the help of Lemma~\ref{lem:swap}, we reroute flow within~$G_{\bar{\spnode}}$ until Properties~\eqref{def:almost-unsplittable:i} and~\eqref{def:almost-unsplittable:ii} are fulfilled for both~$\spnode_1$ and~$\spnode_2$.
First, by repeatedly applying Corollary~\ref{cor:swap} to pairs of commodities~$i_1,i_2\in\fractionalDemands_{\spnode_1}(\vec{\mcflow})\cap\fractionalDemands_{\spnode_2}(\vec{\mcflow})$, we ensure that
\begin{equation}
\label{eq:help1}
|\fractionalDemands_{\spnode_1}(\vec{\mcflow})\cap\fractionalDemands_{\spnode_2}(\vec{\mcflow})|\leq1,
\end{equation}
and Property~\eqref{def:almost-unsplittable:ii} is fulfilled. By Observation~\ref{obs:z-in-series-parallel}\,\eqref{obs:z-in-parallel}, every commodity~$i\notin\fractionalDemands_{\bar{\spnode}}(\vec{\mcflow})$ (i.e., $z_{\bar\spnode,i} (\vec{\mcflow})\in\{0,1\}$) that is routed fractionally through the subgraph~$G_{\spnode_1}$ (i.e., $0<z_{\spnode_1,i} (\vec{\mcflow})<1$) must also be routed fractionally through~$G_{\spnode_2}$; that is,
\begin{equation}
\label{eq:help2}
\fractionalDemands_{\spnode_1}(\vec{\mcflow})\setminus\fractionalDemands_{\bar\spnode}(\vec{\mcflow})=\fractionalDemands_{\spnode_2}(\vec{\mcflow})\setminus\fractionalDemands_{\bar\spnode}(\vec{\mcflow})\subseteq\fractionalDemands_{\spnode_1}(\vec{\mcflow})\cap\fractionalDemands_{\spnode_2}(\vec{\mcflow}).
\end{equation}
Thus, due to~\eqref{eq:help1}, the set~$\fractionalDemands_{\spnode_1}(\vec{\mcflow})\setminus\fractionalDemands_{\bar\spnode}(\vec{\mcflow})=\fractionalDemands_{\spnode_2}(\vec{\mcflow})\setminus\fractionalDemands_{\bar\spnode}(\vec{\mcflow})$ has cardinality~$0$ or~$1$. If the cardinality is~$0$, then both~$\fractionalDemands_{\spnode_1}(\vec{\mcflow})$ and~$\fractionalDemands_{\spnode_2}(\vec{\mcflow})$ are subsets of~$\fractionalDemands_{\bar\spnode}(\vec{\mcflow})$, and Property~\eqref{def:almost-unsplittable:i} holds for both~$\spnode_1$ and~$\spnode_2$.

Otherwise, there is a commodity~$i_1\notin\fractionalDemands_{\bar\spnode}(\vec{\mcflow})$ with~$\fractionalDemands_{\spnode_1}(\vec{\mcflow})\setminus\fractionalDemands_{\bar\spnode}(\vec{\mcflow})=\fractionalDemands_{\spnode_2}(\vec{\mcflow})\setminus\fractionalDemands_{\bar\spnode}(\vec{\mcflow})=\{i_1\}$, and every commodity from~$\fractionalDemands_{\bar\spnode}(\vec{\mcflow})$ is either contained in~$\fractionalDemands_{\spnode_1}(\vec{\mcflow})$ or in~$\fractionalDemands_{\spnode_2}(\vec{\mcflow})$ but not in both. Therefore, Property~\eqref{def:almost-unsplittable:i} is only violated for~$\spnode_1$ or~$\spnode_2$, if~$\fractionalDemands_{\bar\spnode}(\vec{\mcflow})$ consists of two distinct commodities~$i_2,i_3$ that are both contained in either~$\fractionalDemands_{\spnode_1}(\vec{\mcflow})$ or~$\fractionalDemands_{\spnode_2}(\vec{\mcflow})$. By symmetry, we may restrict to the case where~$\fractionalDemands_{\spnode_1}(\vec{\mcflow})=\{i_1\}$ and~$\fractionalDemands_{\spnode_2}(\vec{\mcflow})=\{i_1,i_2,i_3\}$. Applying Lemma~\ref{lem:swap} to commodities~$i_1,i_2$ produces a new multiflow~$\tilde{\vec{\mcflow}}$
that satisfies at least one of the following two conditions: 1) commodity~$i_1$ is no longer routed through~$G_{\spnode_1}$ and thus~$z_{\spnode_2,i_1}=1$, that is,
\[
\fractionalDemands_{\spnode_1}(\tilde{\vec{\mcflow}})=\{i_2\}
\quad\text{and}\quad
\fractionalDemands_{\spnode_2}(\tilde{\vec{\mcflow}})\subseteq\{i_2,i_3\},
\]
or 2) commodity~$i_2$ is no longer routed through~$G_{\spnode_2}$, that is,
\[
\fractionalDemands_{\spnode_1}(\tilde{\vec{\mcflow}})\subseteq\{i_1,i_2\}
\quad\text{and}\quad
\fractionalDemands_{\spnode_2}(\tilde{\vec{\mcflow}})\subseteq\{i_1,i_3\}.
\]
In both cases, Properties~\eqref{def:almost-unsplittable:i} and~\eqref{def:almost-unsplittable:ii} are obviously fulfilled for~$\spnode_1$ and~$\spnode_2$.
\end{proof}

\section{Convex decompositions of multiflows}
\label{sec:convex:combination}

In this section, as an intermediate step towards our main result, \Cref{thm:main}, we first prove a slightly weaker version, allowing unsplittable multiflows to deviate from the given fractional multiflow by at most~$2d_{\max}$ on each arc.\footnote{Traub, Vargas Koch, and Zenklusen obtained the same bound for the special case of a single source node~$s=s_i$, $i\in[k]$, on the more general class of planar and acyclic digraphs; see~\cite[Theorem~1.8]{TraubVargasKochZenklusen2024}.}

\begin{theorem}
\label{thm:2dmax}
For any multiflow~$\vec{\mcflow}$ in~$G$, its total arc flows~$(x_e)_{e\in E}$ can be expressed as a convex combination of total arc flows~$(y_e)_{e\in E}$ of unsplittable multiflows satisfying
\begin{align}
\label{eq:2dmax}
x_e-2d_{\max}<y_e<x_e+2d_{\max}\qquad\text{for every arc~$e\in E$.}
\end{align}
Moreover, the unsplittable multiflows and the coefficients of the convex combination can be computed efficiently.
\end{theorem}

We discuss the proof of \cref{thm:2dmax} in detail since the construction of the convex decomposition also forms the basis of the proof of our main result \cref{thm:main}.
In view of \Cref{thm:almost-unsplittable}, it suffices to prove \Cref{thm:2dmax} for an almost unsplittable multiflow~$\vec{\mcflow}$ with the same total flow~$\vec{\totalflow}$.
Although the statement of \cref{thm:2dmax} is only about the total flows, we proceed by showing the stronger statement that an almost unsplittable multiflow~$\vec{\mcflow}$ itself (i.e., the full matrix $\vec{\mcflow}$ and not only the total flow vector~$\vec{\totalflow}$) can be expressed as a convex combination of unsplittable multiflows.

\subsection{Unsplittable Routings.}
In order to construct suitable unsplittable multiflows for the convex decomposition, we leverage the properties of almost unsplittable multiflows.
If $\vec{\mcflow}$ is almost unsplittable, by \Cref{def:almost-unsplittable}, for each node $\spnode\in V_T$ of the $sp$-tree~$T$, multiflow~$\vec{\mcflow}$ routes (at most) two commodities fractionally through the corresponding component.
To simplify notation and avoid case distinctions, we assume that~$|\fractionalDemands_\spnode| = 2$ for every~$\spnode$, and let~$\ileft,\iright\in [k]$ be the two distinct commodities routed fractionally through~$\spnode$, i.e., $\fractionalDemands_\spnode = \{ \ileft, \iright \}$. We take care of the case~$|\fractionalDemands_\spnode|<2$ by introducing dummy commodities in \cref{app:special-cases}.
Since in an almost unsplittable multiflow all but the two commodities~$\ileft, \iright \in \fractionalDemands_\spnode$ must be routed through the subdigraph~$G_\spnode$, we are only left with four possible routing options.
We represent these options with the subsets
\begin{equation}
J^{\spnode}_1 \coloneqq \emptyset,\qquad J^{\spnode}_2 \coloneqq \{\ileft\},\qquad J^{\spnode}_3 \coloneqq \{\iright\},\qquad\text{and}\quad J^{\spnode}_4 \coloneqq \{ \ileft, \iright \},
\label{eq:sets:j}
\end{equation}
where routing option~$j \in \{1,2,3,4\}$ means that exactly the commodities in $\unsplittableDemands_\spnode \cup J^\spnode_j$ (i.e., the commodities that are completely routed through~$G_\spnode$ plus the commodities selected in the set~$J^\spnode_j$) are routed through the subdigraph~$G_\spnode$. We formalize this in the following definition.
\begin{definition}
	We say an unsplittable multiflow~$\vec{Y}$ \emph{respects the routing option $j \in \{1,2,3,4\}$ in~$\spnode \in V_T$} if
	$
		z_{\spnode, i} (\vec{Y}) = 1
		\text{ if } i \in \unsplittableDemands_\spnode \cup J^{\spnode}_j
	$
	and
	$
		z_{\spnode, i} (\vec{Y}) = 0
		\text{ if } i \notin \unsplittableDemands_\spnode \cup J^{\spnode}_j
		.
	$
\end{definition}
Notice, that the definition above (as well as most definitions in this section) are always in relation to the given (almost unsplittable) multiflow~$\vec{\mcflow}$ since the sets~$\unsplittableDemands_\spnode$ and $J^\spnode_j$ depend on this flow.
For any multiflow~$\vec{X}$ (fractional and unsplittable), we define
\[
    \componentflow_{\spnode, i} \coloneqq z_{\spnode, i} (\vec\mcflow) d_i
    \quad \text{and} \quad
    \componentflow_{\spnode} \coloneqq \sum_{i=1}^k \componentflow_{\spnode, i},
\]
where $\componentflow_{\spnode, i}$ is the amount of flow of some commodity~$i$ routed through the subgraph~$G_\spnode$ and $\componentflow_{\spnode}$ is the total flow of all commodities routed through the subgraph~$G_\spnode$.
In the special case of a leaf node~$\spnode \in V_T$ that represents an arc~$e \in E$, we have $\mcflow_{e, i} = \componentflow_{\spnode, i}$ and $\totalflow_{e} = \componentflow_{\spnode}$.
In view of the desired bound~\eqref{eq:2dmax} we obtain the following result that follows immediately from the previous definition.
\begin{lemma}\label{lem:routing:2dmax}
	Let $\vec{\mcflow}$ be an almost unsplittable flow and $\vec{Y}$ an unsplittable flow.
	\begin{enumerate}[(i)]
		\item If $\vec{Y}$ respects the routing option $j \in \{1,2,3,4\}$ in some $\spnode \in V_T$, then
		$
            |\componentflow[y]_{\spnode} - \componentflow[x]_{\spnode}| < 2 d_{\max}
			.
		$
		\item If, for every (leaf) node~$\spnode \in V_T$, $\vec{Y}$ respects some routing option $j_\spnode \in \{1,2,3,4\}$, then
		\[
			\totalflow_e - 2 d_{\max} <  y_e < \totalflow_e + 2 d_{\max}
		\]
		for all $e \in E$.
	\end{enumerate}
\end{lemma}
\begin{proof}
	Assume that $\vec{Y}$ respects the routing option $j \in \{1,2,3,4\}$ in some fixed $\spnode$. Then,
	\[
        |\componentflow[y]_{\spnode} - \componentflow[x]_{\spnode}|
        =
		\Bigl|\sum_{i \in [k]} z_{\spnode, i} (\vec{Y}) d_i - \sum_{i \in [k]} z_{\spnode, i} (\vec\mcflow) d_i \Bigr|
		=
		\Big| \sum_{i \in J^{\spnode}_j} d_i - \sum_{i \in \fractionalDemands_{\spnode}} z_{\spnode, i} d_i \Big| <
		d_{\ileft} + d_{\iright} \leq 2 d_{\max}
	\]
	proving \emph{(i)}.
	For every arc~$e \in E$ there is some leaf node $\spnode \in V_T$ such that $\componentflow[x]_{\spnode} = \totalflow_e$ and $\componentflow[y]_{\spnode} = y_e$, proving \emph{(ii)}.
\end{proof}

Our goal is to define unsplittable flows~$\vec{Y}_{\ell}, \ell = 1, \dotsc, \numflows$ that respect one of the routing options $j\in\{1,2,3,4\}$ in every~$\spnode \in V_T$ and that are a convex decomposition of the given multiflow~$\vec{\mcflow}$, i.e.,
\begin{equation}\label{eq:unsplittable:convex:combination}
	\sum_{\ell = 1}^\numflows \rho_{\ell} \vec{Y}_{\ell} = \vec{\mcflow}
\end{equation}
for suitable convex coefficients $\rho_\ell \in [0,1]$ with $\sum_{\ell=1}^{\numflows} \rho_\ell = 1$.
If all unsplittable flows~$\vec{Y}_\ell$ respect some routing option for every~$\spnode \in V_T$, we can group the flows with respect to the respective routing options.
Formally, let $(\vec{Y}_\ell)_{\ell \in [\numflows]}$ be a family of unsplittable multiflows that all respect some routing option in every~$\spnode \in V_T$ and assume that \cref{eq:unsplittable:convex:combination} holds. Then, we define for every $\spnode \in V_T$ groups $\flowgroup^\spnode_j \coloneqq \{ \ell \in [\numflows] \mid \vec{Y}_\ell \text{ respects routing option }j \text{ in }\spnode\}, j=1,2,3,4$.
Additionally, we denote by $\mu^{\spnode}_j \coloneqq \sum_{\ell \in K^{\spnode}_j} \rho_{\ell}$ the total weight on group~$j \in \{1,2,3,4\}$.
We can now characterize suitable convex decompositions as follows.
\begin{lemma}\label{lem:mu:consistency}
    Let $\vec{\mcflow}$ be a given almost unsplittable multiflow. Further, let~$(\vec{Y}_{\ell})_{\ell \in [\numflows]}$ be a family of unsplittable multiflows that respect some routing option~$j \in \{1,2,3,4\}$ in every node~$\spnode \in V_T$ of the $sp$-tree and let~$(\rho_{\ell})_{\ell \in [\numflows]}$ be a family of coefficients~$\rho_{\ell} \in [0,1]$ with $\sum_{l=1}^{\numflows} \rho_{\ell} = 1$.
    Then, these families satisfy \cref{eq:unsplittable:convex:combination} if and only if
    \begin{equation}
        \mu^\spnode_2+\mu^\spnode_4 = z_{\spnode,\ileft},\quad \mu^\spnode_3+\mu^\spnode_4 = z_{\spnode,\iright},\quad\text{and}\quad \mu^\spnode_1+\mu^\spnode_2+\mu^\spnode_3+\mu^\spnode_4=1
        \label{eq:mu:conditions}
    \end{equation}
    holds for every node~$\spnode \in V_T$,
    where $z_{\spnode,\ileft} \coloneqq z_{\spnode,\ileft} (\vec{\mcflow})$ and $z_{\spnode,\iright} \coloneqq z_{\spnode,\iright} (\vec{\mcflow})$ are the respective demand shares of the given multiflow~$\vec{\mcflow}$.
\end{lemma}
\begin{proof}
First, we note that the last equation in~\eqref{eq:mu:conditions} is equivalent to the condition~$\sum_{\ell=1}^\numflows \rho_{\ell} = 1$.
\Cref{eq:unsplittable:convex:combination} can be rewritten component-wise as
$
    \sum_{\ell=1}^{\numflows} \rho_{\ell} Y_{\ell, e, i} = \mcflow_{e, i}
$
for every $e \in E$ and $i \in [k]$. Assume that~\eqref{eq:unsplittable:convex:combination} holds. Fix some node~$\spnode \in V_T$ and some commodity~$i \in [k]$ and sum this equation over all arcs $e \in \delta^+_\spnode (u_{\spnode})$. Then, we get from the left-hand side that
\begin{align*}
    \sum_{e \in \delta^{+}_\spnode (u_{\spnode})} \sum_{\ell=1}^{\numflows} \rho_{\ell} Y_{\ell, e, i}
    &= \sum_{\ell=1}^{\numflows} \rho_{\ell} z_{\spnode, i} (\vec{Y}_{\ell}) d_i
    = d_i \sum_{\ell: z_{\spnode, i} (\vec{Y}_{\ell}) = 1} \rho_{\ell}
    =
    \begin{cases}
        1 \cdot d_i &\text{if } i \in \unsplittableDemands_{\spnode}, \\
        (\mu^{\spnode}_2 + \mu^{\spnode}_4) d_{i} &\text{if } i = \ileft, \\
        (\mu^{\spnode}_3 + \mu^{\spnode}_4) d_{i} &\text{if } i = \iright, \\
        0 &\text{otherwise.}
    \end{cases}
\end{align*}
From the right-hand side, we get $\sum_{e \in \delta^{+}_\spnode (u_{\spnode})} \mcflow_{e,i} = z_{\spnode, i} (\vec{\mcflow}) d_i$. This implies~\eqref{eq:mu:conditions}.

On the other hand, if the families satisfy~\eqref{eq:mu:conditions}, we can use the above computation for the special case of leaf nodes~$\spnode$ that represent a fixed arc~$e \in E$. Since $\mcflow_{e,i} = z_{\spnode, i} (\vec{\mcflow}) d_i$ in this case, and $ z_{\spnode, i} = 1$ if $i \in \unsplittableDemands_{\spnode}$ and $ z_{\spnode, i} = 0$ if $i \notin \unsplittableDemands_{\spnode} \cup \{\ileft, \iright\}$, \eqref{eq:unsplittable:convex:combination} follows.
\end{proof}
The next lemma gives an explicit solution to the linear system~\eqref{eq:mu:conditions}. For a concise representation, we denote for any number $z\in \mathbb{R}$ by~$[z]^+ $ the \emph{positive part} of~$z$, that is,
$[z]^+ \coloneqq z$ if $z > 0$, and $[z]^+\coloneqq 0$ if $z \leq 0$.
The proof relies on the following basic properties of the positive part. For any two real numbers $a, b \in \R$, we have
\begin{align}
	[a - b]^+ - \max\{a,b\} &= -b,
	&
	[a - b]^+ + \min\{a,b\} &= a
	.
	\label{eq:properties:positive:part}
\end{align}

\begin{lemma}\label{lem:mu-values}
	For every node~$\spnode$ of the $sp$-tree, the numbers
	\begin{align}
		\mu^{\spnode}_1 &\coloneqq  \big[(1 - z_{\spnode, \iright}) - z_{\spnode, \ileft} \big]^+,
		&\mu^{\spnode}_2 &\coloneqq \min \{ z_{\spnode, \ileft}, 1 - z_{\spnode, \iright} \} ,\notag\\
		\;
		\mu^{\spnode}_3 &\coloneqq 1 -  \max \{ z_{\spnode, \ileft}, 1 - z_{\spnode, \iright} \} ,
		\;
		&\mu^{\spnode}_4 &\coloneqq 
		\big[ z_{\spnode, \ileft} - (1 - z_{\spnode, \iright}) \big]^+
		\label{eq:definition:mu}
	\end{align}
	are a solution to the linear system~\eqref{eq:mu:conditions}. Moreover, $\mu^\spnode_j \in [0,1], j=1,\dotsc,4$.
\end{lemma}
\begin{proof}
	Since $z_{\spnode, \ileft}, z_{\spnode, \iright} \in [0, 1]$, $\mu^{\spnode}_j \in [0,1]$ follows for $j = 1, \dotsc, 4$.
	Further, we compute
	\begin{align*}
		\sum_{j=1}^4 \mu^{\spnode}_j &=
		(\mu^{\spnode}_1 + \mu^{\spnode}_3) + (\mu^{\spnode}_2 + \mu^{\spnode}_4) \stackrel{\eqref{eq:properties:positive:part}}{=} (1 - z_{\spnode, \ileft}) + z_{\spnode, \ileft} = 1,
		\\
		\mu^\spnode_2+\mu^\spnode_4 &\stackrel{\mathclap{\eqref{eq:properties:positive:part}}}{=} z_{\spnode, \ileft}, \\
		\mu^\spnode_3+\mu^\spnode_4 &\stackrel{\mathclap{\eqref{eq:properties:positive:part}}}{=} 1 - (1 - z_{\spnode, \iright}) = z_{\spnode, \iright}
	\end{align*}
	concluding the proof.
\end{proof}

Note that, by definition, at most three of the four coefficients defined in~\eqref{eq:definition:mu} are non-zero. Thus, there are at most three groups of unsplittable multiflows with non-zero weight for each node~$\spnode \in V_T$.
This means, in every node~$\spnode$, we will use at most three out of the four routing options.

\begin{example}\label{ex:mus}
	Consider once again \cref{example}. For the node~$\spnode = \spnode_6$ of the $sp$-tree, we have
	$
		\vec{z}^{\spnode} (\vec{\mcflow}) = (0, \tfrac{3}{8}, 1, 1, 1, \tfrac{1}{4}, 0, 0)
		.
	$
	The fractionally routed commodities are $p^{\spnode} = 2$  with $z_{\spnode, p^{\spnode}} = \tfrac{3}{8}$ and $q^\spnode = 6$ with $z_{\spnode, q^{\spnode}} = \tfrac{1}{4}$. Hence, we obtain the coefficients
	\begin{align*}
		\mu_1^{\spnode} = \tfrac{3}{8}
		, \quad
		\mu_2^{\spnode} = \tfrac{3}{8}
		, \quad
		\mu_3^{\spnode} = \tfrac{1}{4}
		, \quad
		\mu_4^{\spnode} = 0
	\end{align*}
	corresponding to the four unsplittable routings
	\[
		\unsplittableDemands_{\spnode} \cup J^{\spnode}_1 = \{3, 4, 5\}, \;
		\unsplittableDemands_{\spnode} \cup J^{\spnode}_2 = \{2, 3, 4, 5\}, \;
		\unsplittableDemands_{\spnode} \cup J^{\spnode}_3 = \{3, 4, 5, 6\}, \;
		\unsplittableDemands_{\spnode} \cup J^{\spnode}_4 = \{2, 3, 4, 5, 6\},
	\]
	where the completely routed commodities $3, 4, 5$ are always routed through~$\spnode_6$. Since the demands are $d_2 = d_6 = 2$ and $d_3 = d_4 = d_5 = 1$ these routings induce total flows of $3$, $5$, $5$, and $7$, respectively. The total flow through the component represented by~$\spnode_6$ given in \cref{example} is $4.25$ (since $\spnode_6$ represents the parallel composition of the arcs~$e_3$ and~$e_4$ with flows $x_{e_3} = 2.25$ and $x_{e_4} = 2$). And, in fact, we have
	$
		4.25 = \tfrac{3}{8} \cdot 3 + \tfrac{3}{8} \cdot 5 + \tfrac{1}{4} \cdot 5 + 0 \cdot 7
	$.
	Notice that this only gives us information on how the weight of the convex decomposition must be distributed over the groups~$\flowgroup^{\spnode}_j$, but we do not immediately obtain the new convex decomposition with the respective unsplittable flows and the coefficients.
	This can only be obtained by combining and refining the convex decompositions of the child nodes.
\end{example}

\subsection{Recursive Refinements of Convex Decompositions}
We construct a convex decomposition of a given, almost unsplittable multiflow~$\vec{\mcflow}$ recursively over the nodes~$\spnode \in V_T$ of the $sp$-tree.
Formally, we show the following statement.
\begin{lemma}\label{lem:recursive:convex}
	Let $\vec{\mcflow}$ be an almost unsplittable multiflow. Then, for every node~$\spnode \in V_T$ of the $sp$-tree, there exist families of unsplittable multiflows $(\vec{Y}^{\spnode}_{\ell})_{\ell=1, \dotsc, \numflows_{\spnode}}$ in the subgraph~$G_{\spnode}$ and coefficients $(\rho^{\spnode}_{\ell})_{\ell=1, \dotsc, \numflows_{\spnode}}$ with the following properties.
	\begin{enumerate}[(i)]
		\item For every $\spnode \in V_T$,
			\[
			\vec{\mcflow}^{\spnode} = \sum_{\ell=1}^{\numflows_{\spnode}} \rho^{\spnode}_{\ell} \vec{Y}^{\spnode}_{\ell}
			,
			\]
			where $\vec{\mcflow}^{\spnode}$ is the restriction of the flow~$\vec{\mcflow}$ to the subgraph~$G_{\spnode}$.
		\item For every $\spnode \in V_T$, every flow~$\vec{Y}^{\spnode}_{\ell}, \ell= 1, \dotsc, \numflows_{\spnode}$ respects some routing option $j \in \{1,2,3,4\}$ in all nodes~$\spnode'$ that are descendants of~$\spnode$ in the $sp$-tree (including $\spnode$ itself).
	\end{enumerate}
\end{lemma}

Eventually, we can use \cref{lem:recursive:convex} for $\spnode = \spnode_0$ (where $\spnode_0$ is the root of the $sp$-tree) in order to prove that the desired convex decomposition in \cref{thm:2dmax} exists.
We continue by proving \cref{lem:recursive:convex} by induction over the nodes~$\spnode \in V_T$, starting from the leaf nodes.

\subsubsection{Convex decompositions in \texorpdfstring{$q$}{q}-nodes}
Let $\spnode \in V_T$ be a leaf node of the $sp$-tree, i.e., $\spnode$ represents some arc~$e \in E$. Since the subgraph~$G_{\spnode}$ is only the single arc~$e$, the desired flows in \cref{lem:recursive:convex} can be defined by only defining the flow on this single arc.
We define four unsplittable multiflows~$\vec{Y}^{\spnode}_{\ell}, \ell=1,2,3,4$ in $G_{\spnode}$ and the corresponding coefficients~$\rho_{\ell}, \ell=1,2,3,4$ by
\begin{align*}
	Y_{\ell, e, i} \coloneqq
	\begin{cases}
		d_i & \text{if } i \in \unsplittableDemands_\spnode \cup J^{\spnode}_{\ell}, \\
		0 & \text{otherwise}
	\end{cases}
	\quad \text{and} \quad
	\rho_{\ell} \coloneqq \mu^{\spnode}_{\ell}
	\qquad \text{for } \ell=1,2,3,4,
\end{align*}
where the numbers~$\mu^{\spnode}_{\ell}$ are the numbers defined in \eqref{eq:definition:mu}.
By definition, the flow $\vec{Y}^{\spnode}_{\ell}$ respects the routing option~$\ell$ in $\spnode$. Since~$\spnode$ is a leaf node, there are no further descendants in the~$sp$-tree and, hence, the flows satisfy property~\emph{(ii)} of \cref{lem:recursive:convex} for the node~$\spnode$. Since every flow group~$\flowgroup^{\spnode}_j$ consists of exactly one flow weighted by~$\mu^{\spnode}_j$, \cref{lem:mu:consistency} implies that property~\emph{(i)} holds as well.

\subsubsection{Convex refinements in \texorpdfstring{$s$}{s}-nodes}
Let $\spnode \in V_T$ be an $s$-node of the $sp$-tree. Further, assume (by induction) that the statement of \cref{lem:recursive:convex} is satisfied for both child nodes~$\spnode_1$ and~$\spnode_2$. By this assumption, there are families~$(\vec{Y}^{\spnode_1}_{\ell})_{\ell=1, \dotsc, \numflows_{\spnode_1}}$ and $(\rho^{\spnode_1}_{\ell})_{\ell=1, \dotsc, \numflows_{\spnode_1}}$ as well as $(\vec{Y}^{\spnode_2}_{\ell})_{\ell=1, \dotsc, \numflows_{\spnode_2}}$ and $(\rho^{\spnode_2}_{\ell})_{\ell=1, \dotsc, \numflows_{\spnode_2}}$ satisfying the conditions \cref{lem:recursive:convex}.
We proceed to construct a new convex decomposition for the $s$-node~$\spnode$ by combining the flows of the child nodes.
For an $s$-node this procedure is straightforward. First, we observe that since the subgraphs $G_{\spnode_1}$ and $G_{\spnode_2}$ are arc-disjoint, we can combine any two flows in these subgraphs easily by just combining the respective flow matrices.
By \cref{obs:z-in-series-parallel}\ref{obs:z-in-serial}, we have $\unsplittableDemands_{\spnode_1} = \unsplittableDemands_{\spnode_2} = \unsplittableDemands_{\spnode}$. Therefore, the routing options within these three nodes are exactly the same.
If we only combine flows $\vec{Y}^{\spnode_1}_{\ell_1}$ and $\vec{Y}^{\spnode_2}_{\ell_2}$ that both respect the same routing option $j \in \{1,2,3,4\}$, the resulting flows are feasible (in the subgraph~$G_{\spnode}$) and also respect the same routing option~$j \in \{1,2,3,4\}$.

Formally, we can use \cref{alg:linear:combination:refinement} from \cref{app:refinements} with~$\bar{\rho} = \mu^{\spnode}_j$ (from \eqref{eq:definition:mu}) to refine the non-negative linear combinations $\sum_{\ell \in \flowgroup^{\spnode_1}_j} \rho^{\spnode_1}_{\ell} \vec{Y}^{\spnode_1}_{\ell}$ and $\sum_{\ell \in \flowgroup^{\spnode_2}_j} \rho^{\spnode_2}_{\ell} \vec{Y}^{\spnode_2}_{\ell}$ of flows in the respective groups~$\flowgroup^{\spnode_1}_j$ and~$\flowgroup^{\spnode_2}_j$ yielding a non-negative linear combination of flows in group~$\flowgroup^{\spnode}_j$ for the parent node~$\spnode$.
By this construction, all flows in group~$\flowgroup^{\spnode}_j$ are weighted with $\bar{\rho} = \mu^{\spnode}_j$. Therefore, \cref{lem:mu:consistency} implies that the new convex decomposition satisfies~\emph{(i)} from \cref{lem:recursive:convex} for the parent node~$\spnode$. Property~\emph{(ii)} is satisfied by definition for~$\spnode$ and inherited from the flows from the child nodes for all descendants.

The following example illustrates the procedure described above.

\begin{example} \label{s-example}
	We introduce an example for the combination of the convex decompositions in an $s$-node~$\spnode$ with child nodes $\spnode_1$ and $\spnode_2$.
	For this example, we assume that we are given the convex decompositions of the child nodes. In particular, we assume that the convex combinations for child~$\spnode_1$ is as follows.
	\begin{align*}
		{\vec{\mcflow}}^{\spnode_1} &=
		\sum_{\ell = 1}^{\numflows_{\spnode_1}} \rho^{\spnode_1}_\ell \vec{Y}^{\spnode_1}_\ell
		=
		\colorbrace{color1}{
		\tfrac{1}{10} \vec{Y}^{\spnode_1}_1
		+
		\tfrac{1}{10} \vec{Y}^{\spnode_1}_2
		}{\substack{\text{group }\flowgroup^{\spnode_1}_1 \\ \mu^{\spnode_1}_1 = \frac{1}{10} + \frac{1}{10} = \frac{1}{5} }}
		+
		\colorbrace{color2}{
		\tfrac{1}{10} \vec{Y}^{\spnode_1}_3 +
		\tfrac{3}{10} \vec{Y}^{\spnode_1}_4
		}{\substack{\text{group }\flowgroup^{\spnode_1}_2 \\ \mu^{\spnode_1}_2 = \frac{1}{10} + \frac{3}{10} = \frac{2}{5} }}
		+
		\colorbrace{color3}{
		\tfrac{2}{25} \vec{Y}^{\spnode_1}_5 +
		\tfrac{4}{25} \vec{Y}^{\spnode_1}_6 +
		\tfrac{4}{25} \vec{Y}^{\spnode_1}_7
		}{\substack{\text{group }\flowgroup^{\spnode_1}_3 \\ \mu^{\spnode_1}_3 = \frac{2}{25} + \frac{4}{25} + \frac{4}{25} = \frac{2}{5} }}
	\intertext{For child~$\spnode_2$, we assume that the convex combinations is as follows.}
		{\vec{\mcflow}}^{\spnode_2} &=
		\sum_{\ell = 1}^{\numflows_{\spnode_2}} \rho^{\spnode_2}_\ell \vec{Y}^{\spnode_2}_\ell
		=
		\colorbrace{color5}{
		\tfrac{1}{20} \vec{Y}^{\spnode_2}_1
		+
		\tfrac{3}{20} \vec{Y}^{\spnode_2}_2
		}{\substack{\text{group }\flowgroup^{\spnode_2}_1 \\ \mu^{\spnode_2}_1 = \frac{1}{20} + \frac{3}{20} = \frac{1}{5} }}
		+
		\colorbrace{color6}{
		\tfrac{3}{10} \vec{Y}^{\spnode_2}_3 +
		\tfrac{1}{10} \vec{Y}^{\spnode_2}_4
		}{\substack{\text{group }\flowgroup^{\spnode_2}_2 \\ \mu^{\spnode_2}_2 = \frac{3}{10} + \frac{1}{10} = \frac{2}{5} }}
		+
		\colorbrace{color7}{
		\tfrac{4}{25} \vec{Y}^{\spnode_2}_5 +
		\tfrac{6}{25} \vec{Y}^{\spnode_2}_6
		}{\substack{\text{group }\flowgroup^{\spnode_2}_3 \\ \mu^{\spnode_2}_3 = \frac{4}{25} + \frac{6}{25} = \frac{2}{5} }}
	\end{align*}
	Both convex decompositions are depicted in the upper part of \cref{fig:refinements:serial}. Notice, for both convex decomposition of the child  nodes, there is no flow in the fourth group. In fact, we know from \cref{lem:mu-values}, that at least one group has weight zero which is assumed to be the fourth group in this example. Additionally, notice that the total weights of the groups are the same, i.e. $\mu^{\spnode_1}_{j} = \mu^{\spnode_2}_{j}$ for $j=1,2,3,4$ which is implied by \cref{obs:z-in-series-parallel}.

	In the parent $s$-node~$\spnode$, the flows of the respective groups of the children are combined with \cref{alg:linear:combination:refinement} yielding the following linear combinations for the groups, where $\begin{bmatrix} \vec{Y}^{\spnode_1}_{\ell_1} \\ \vec{Y}^{\spnode_2}_{\ell_2}  \end{bmatrix}$ denotes the combined flow matrix if the flows $\vec{Y}^{\spnode_1}_{\ell_1}$ and $\vec{Y}^{\spnode_2}_{\ell_2}$ are combined.
	\begin{align*}
		\tilde{\vec{\mcflow}}^{\spnode, 1} &=
		\sum_{\ell=1}^3 \rho^{\spnode}_\ell \vec{Y}^{\spnode}_\ell
		=
		\frac{1}{20} \begin{bmatrix} \vec{Y}^{\spnode_1}_1 \\ \vec{Y}^{\spnode_2}_1  \end{bmatrix} +
		\frac{1}{20} \begin{bmatrix} \vec{Y}^{\spnode_1}_1 \\ \vec{Y}^{\spnode_2}_2 \end{bmatrix} +
		\frac{1}{10} \begin{bmatrix} \vec{Y}^{\spnode_1}_2 \\ \vec{Y}^{\spnode_2}_2  \end{bmatrix} \\
		\tilde{\vec{\mcflow}}^{\spnode, 2} &=
		\sum_{\ell=3}^5 \rho^{\spnode}_\ell \vec{Y}^{\spnode}_\ell
		=
		\frac{1}{10} \begin{bmatrix} \vec{Y}^{\spnode_1}_3 \\ \vec{Y}^{\spnode_2}_3  \end{bmatrix} +
		\frac{2}{10} \begin{bmatrix} \vec{Y}^{\spnode_1}_4 \\ \vec{Y}^{\spnode_2}_3  \end{bmatrix} +
		\frac{1}{10} \begin{bmatrix} \vec{Y}^{\spnode_1}_4 \\ \vec{Y}^{\spnode_2}_4  \end{bmatrix} \\
		\tilde{\vec{\mcflow}}^{\spnode, 3} &=
		\sum_{\ell=6}^{10} \rho^{\spnode}_\ell \vec{Y}^{\spnode}_\ell
		=
		\frac{2}{25} \begin{bmatrix} \vec{Y}^{\spnode_1}_5 \\ \vec{Y}^{\spnode_2}_5  \end{bmatrix} +
		\frac{2}{25} \begin{bmatrix} \vec{Y}^{\spnode_1}_6 \\ \vec{Y}^{\spnode_2}_5  \end{bmatrix} +
		\frac{2}{25} \begin{bmatrix} \vec{Y}^{\spnode_1}_6 \\ \vec{Y}^{\spnode_2}_6  \end{bmatrix} +
		\frac{4}{25} \begin{bmatrix} \vec{Y}^{\spnode_1}_7 \\ \vec{Y}^{\spnode_2}_6  \end{bmatrix}
	\end{align*}
	Overall, we obtain the convex decomposition
	\[
		\sum_{\ell=1}^{10} \rho^{\spnode}_\ell \vec{Y}^{\spnode}_\ell
		=
		\tilde{\vec{\mcflow}}^{\spnode, 1} + \tilde{\vec{\mcflow}}^{\spnode, 2} + \tilde{\vec{\mcflow}}^{\spnode, 3}
		=
		\begin{bmatrix}
			\sum_{\ell = 1}^{\numflows_{\spnode_1}} \rho^{\spnode_1}_\ell \vec{Y}^{\spnode_1}_\ell \\
			\sum_{\ell = 1}^{\numflows_{\spnode_2}} \rho^{\spnode_2}_\ell \vec{Y}^{\spnode_2}_\ell
		\end{bmatrix}
		=
		\begin{bmatrix}
			\vec{\mcflow}^{\spnode_1} \\
			\vec{\mcflow}^{\spnode_2}
		\end{bmatrix}
		=
		\vec{\mcflow}^{\spnode}
	\]

\end{example}

\def\convboxwidth{7}
\def\convboxheight{1em}
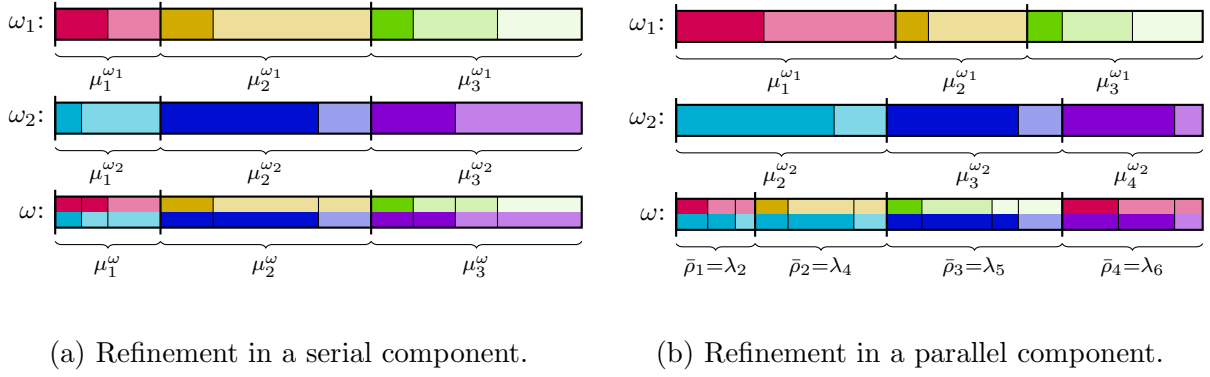
\begin{figure}[t]%
	\begin{subfigure}{.5\textwidth}
		\begin{center}
		\begin{tikzpicture}[xscale={\linewidth/8cm}]

			\node[anchor=east] at (0,.5*\convboxheight) {\footnotesize$\spnode_1$:};

			\foreach \subdivision/\mustart/\muend/\lbl in {
				{0/.5/color1,.5/1/color1!50}/{0}/{.2}/{\mu^{\spnode_1}_1},%
				{0/.25/color2,.25/1/color2!40}/{.2}/{.6}/{\mu^{\spnode_1}_2},%
				{0/.2/color3,.2/.6/color3!30,.6/1/color3!10}/{.6}/{1}/{\mu^{\spnode_1}_3}%
			} {
				\foreach \alphastart/\alphaend/\col in \subdivision {
					\fill[\col] ({(\mustart+\alphastart*(\muend-\mustart))*\convboxwidth}, 0) rectangle ++({(\muend-\mustart)*(\alphaend-\alphastart)*\convboxwidth},\convboxheight);
					\draw ({(\mustart+\alphastart*(\muend-\mustart))*\convboxwidth}, 0) -- ++(0, \convboxheight);
				}

				\draw[thick] (\mustart*\convboxwidth, -.5ex) -- ++(0, \convboxheight+1ex);
				\draw[decorate, decoration={brace, mirror}] ({\mustart*\convboxwidth}, -1ex) -- ++({(\muend - \mustart)*\convboxwidth}, 0) node[midway, below=.25ex] {\scriptsize $\lbl$};
			}

			\draw[thick] (0,0) rectangle (\convboxwidth, \convboxheight);


			\begin{scope}[shift={(0,-1.25)}]
			\node[anchor=east] at (0,.5*\convboxheight) {\footnotesize$\spnode_2$:};

			\foreach \subdivision/\mustart/\muend/\lbl in {
				{0/.25/color5,.25/1/color5!50}/{0}/{.2}/{\mu^{\spnode_2}_1},%
				{0/.75/color6,.75/1/color6!40}/{.2}/{.6}/{\mu^{\spnode_2}_2},%
				{0/.4/color7,.4/1/color7!50}/{.6}/{1}/{\mu^{\spnode_2}_3}%
			} {
				\foreach \alphastart/\alphaend/\col in \subdivision {
					\fill[\col] ({(\mustart+\alphastart*(\muend-\mustart))*\convboxwidth}, 0) rectangle ++({(\muend-\mustart)*(\alphaend-\alphastart)*\convboxwidth},\convboxheight);
					\draw ({(\mustart+\alphastart*(\muend-\mustart))*\convboxwidth}, 0) -- ++(0, \convboxheight);
				}

				\draw[thick] (\mustart*\convboxwidth, -.5ex) -- ++(0, \convboxheight+1ex);
				\draw[decorate, decoration={brace, mirror}] ({\mustart*\convboxwidth}, -1ex) -- ++({(\muend - \mustart)*\convboxwidth}, 0) node[midway, below=.25ex] {\scriptsize $\lbl$};
			}

			\draw[thick] (0,0) rectangle (\convboxwidth, \convboxheight);

			\end{scope}


			\begin{scope}[shift={(0,-2.5)}]
				\node[anchor=east] at (0,.5*\convboxheight) {\footnotesize$\spnode$:};

				\foreach \subdivisionA/\subdivisionB/\mustart/\muend/\lbl in {
					{0/.5/color1,.5/1/color1!50}/{0/.25/color5,.25/1/color5!50}/{0}/{.2}/{\mu^{\spnode}_1},%
					{0/.25/color2,.25/1/color2!40}/{0/.75/color6,.75/1/color6!40}/{.2}/{.6}/{\mu^{\spnode}_2},%
					{0/.2/color3,.2/.6/color3!30,.6/1/color3!10}/{0/.4/color7,.4/1/color7!50}/{.6}/{1}/{\mu^{\spnode}_3}%
				} {

					\foreach \alphastart/\alphaend/\col in \subdivisionA {
						\fill[\col] ({(\mustart+\alphastart*(\muend-\mustart))*\convboxwidth}, 0.5*\convboxheight) rectangle ++({(\muend-\mustart)*(\alphaend-\alphastart)*\convboxwidth},0.5*\convboxheight);
					}
					\foreach \alphastart/\alphaend/\col in \subdivisionB {
						\fill[\col] ({(\mustart+\alphastart*(\muend-\mustart))*\convboxwidth}, 0) rectangle ++({(\muend-\mustart)*(\alphaend-\alphastart)*\convboxwidth},0.5*\convboxheight);
						\draw ({(\mustart+\alphastart*(\muend-\mustart))*\convboxwidth}, 0) -- ++(0, \convboxheight);
					}
					\foreach \alphastart/\alphaend/\col in \subdivisionA {
						\draw ({(\mustart+\alphastart*(\muend-\mustart))*\convboxwidth}, 0) -- ++(0, \convboxheight);
					}

					\draw[thick] (\mustart*\convboxwidth, -.5ex) -- ++(0, \convboxheight+1ex);
					\draw[decorate, decoration={brace, mirror}] ({\mustart*\convboxwidth}, -1ex) -- ++({(\muend - \mustart)*\convboxwidth}, 0) node[midway, below=.25ex] {\scriptsize $\lbl$};
				}

				\draw[thick] (0,0) rectangle (\convboxwidth, \convboxheight);

				\end{scope}
		\end{tikzpicture}
		\end{center}
		\caption{Refinement in a serial component.}\label{fig:refinements:serial}
	\end{subfigure}%
	\begin{subfigure}{.5\textwidth}
		\begin{center}
		\begin{tikzpicture}[xscale={\linewidth/8cm}]
            \node[anchor=east] at (0,.5*\convboxheight) {\footnotesize$\spnode_1$:};

            \newcounter{muCounter} 
            \setcounter{muCounter}{1} 

            \foreach \subdivision/\mustart/\muend in {
            {0/.4/color1,.4/1/color1!50}/{0}/{5/12},%
            {0/.25/color2,.25/1/color2!40}/{5/12}/{2/3},%
            {0/.2/color3,.2/.6/color3!30,.6/1/color3!10}/{2/3}/{1}%
            } {
            \gdef\currentMu{\mu_{\arabic{muCounter}}^{\spnode_1}}

            \foreach \alphastart/\alphaend/\col in \subdivision {
                \fill[\col] ({(\mustart+\alphastart*(\muend-\mustart))*\convboxwidth}, 0)
                    rectangle ++({(\muend-\mustart)*(\alphaend-\alphastart)*\convboxwidth},\convboxheight);
                \draw ({(\mustart+\alphastart*(\muend-\mustart))*\convboxwidth}, 0) -- ++(0, \convboxheight);
            }

            \draw[thick] (\mustart*\convboxwidth, -.5ex) -- ++(0, \convboxheight+1ex);
            \draw[decorate, decoration={brace, mirror}] ({\mustart*\convboxwidth}, -1ex)
                -- ++({(\muend - \mustart)*\convboxwidth}, 0) node[midway, below=.25ex]
                {\scriptsize $\currentMu$};

            \stepcounter{muCounter}
            }

            \draw[thick] (0,0) rectangle (\convboxwidth, \convboxheight);


            \begin{scope}[shift={(0,-1.25)}]
            \node[anchor=east] at (0,.5*\convboxheight) {\footnotesize$\spnode_2$:};

            \newcounter{muCounter2} 
            \setcounter{muCounter2}{2} 

            \foreach \subdivision/\mustart/\muend in {
            {0/.75/color5,.75/1/color5!50}/{0}/{2/5},%
            {0/.75/color6,.75/1/color6!40}/{2/5}/{11/15},%
            {0/.8/color7,.8/1/color7!50}/{11/15}/{1}%
            } {
            \gdef\currentMu{\mu_{\arabic{muCounter2}}^{\spnode_2}}

            \foreach \alphastart/\alphaend/\col in \subdivision {
                \fill[\col] ({(\mustart+\alphastart*(\muend-\mustart))*\convboxwidth}, 0)
                    rectangle ++({(\muend-\mustart)*(\alphaend-\alphastart)*\convboxwidth},\convboxheight);
                \draw ({(\mustart+\alphastart*(\muend-\mustart))*\convboxwidth}, 0) -- ++(0, \convboxheight);
            }

            \draw[thick] (\mustart*\convboxwidth, -.5ex) -- ++(0, \convboxheight+1ex);
            \draw[decorate, decoration={brace, mirror}] ({\mustart*\convboxwidth}, -1ex)
                -- ++({(\muend - \mustart)*\convboxwidth}, 0) node[midway, below=.25ex]
                {\scriptsize $\currentMu$};

            \stepcounter{muCounter2}
            }

            \draw[thick] (0,0) rectangle (\convboxwidth, \convboxheight);

            \end{scope}


            \begin{scope}[shift={(0,-2.5)}]
            \node[anchor=east] at (0,.5*\convboxheight) {\footnotesize$\spnode$:};

            \newcounter{rhoCounter} 
            \setcounter{rhoCounter}{1} 

            \foreach \subdivisionA/\subdivisionB/\mustart/\muend/\lbl in {
            {0/.4/color1,.4/1/color1!50}/{0/.75/color5,.75/1/color5!50}/{0}/{3/20}/{\bar{\rho}_1 {=} \lambda_2},%
            {0/.25/color2,.25/1/color2!40}/{0/.75/color5,.75/1/color5!50}/{3/20}/{2/5}/{\bar{\rho}_2 {=} \lambda_4},%
            {0/.2/color3,.2/.6/color3!30,.6/1/color3!10}/{0/.75/color6,.75/1/color6!40}/{2/5}/{11/15}/{\bar{\rho}_3 {=} \lambda_5},%
            {0/.4/color1,.4/1/color1!50}/{0/.8/color7,.8/1/color7!50}/{11/15}/{1}/{\bar{\rho}_4 {=} \lambda_6}%
            } {
            \gdef\currentRho{\bar{\rho}_{\arabic{rhoCounter}}}

            \foreach \alphastart/\alphaend/\col in \subdivisionA {
                \fill[\col] ({(\mustart+\alphastart*(\muend-\mustart))*\convboxwidth}, 0.5*\convboxheight)
                    rectangle ++({(\muend-\mustart)*(\alphaend-\alphastart)*\convboxwidth},0.5*\convboxheight);
            }
            \foreach \alphastart/\alphaend/\col in \subdivisionB {
                \fill[\col] ({(\mustart+\alphastart*(\muend-\mustart))*\convboxwidth}, 0)
                    rectangle ++({(\muend-\mustart)*(\alphaend-\alphastart)*\convboxwidth},0.5*\convboxheight);
                \draw ({(\mustart+\alphastart*(\muend-\mustart))*\convboxwidth}, 0) -- ++(0, \convboxheight);
            }
            \foreach \alphastart/\alphaend/\col in \subdivisionA {
                \draw ({(\mustart+\alphastart*(\muend-\mustart))*\convboxwidth}, 0) -- ++(0, \convboxheight);
            }

            \draw[thick] (\mustart*\convboxwidth, -.5ex) -- ++(0, \convboxheight+1ex);
            \draw[decorate, decoration={brace, mirror}] ({\mustart*\convboxwidth}, -1ex)
                -- ++({(\muend - \mustart)*\convboxwidth}, 0) node[midway, below=.25ex]
                {\scriptsize $\lbl$};

            \stepcounter{rhoCounter}
            }

            \draw[thick] (0,0) rectangle (\convboxwidth, \convboxheight);

            \end{scope}

		\end{tikzpicture}
		\end{center}
		\caption{Refinement in a parallel component.}\label{fig:refinements:parallel}
	\end{subfigure}%
	\caption{Representation of the refinement of convex decompositions in $s$- and $p$-nodes.
	In both subfigures, every row represents the convex decomposition in the respective node. Every color (red, orange, green, blue \dots) represents a group~$\flowgroup_j$ of flows, while the different shades represent different flows within these groups. The length of the rectangles indicates the respective weight~$\rho_\ell$ in the convex combination. In the row corresponding to the parent node~$\spnode$, the combination of colors indicates which flows are combined.}\label{fig:refinements}
\end{figure}

\subsubsection{Convex refinements in \texorpdfstring{$p$}{p}-nodes}

Let $\spnode \in V_T$ be a $p$-node and again assume (by induction) that the statement of \cref{lem:recursive:convex} is satisfied for both child nodes~$\spnode_1$ and~$\spnode_2$. As for $s$-node, we also need to combine the families~$(\vec{Y}^{\spnode_1}_{\ell})_{\ell=1, \dotsc, \numflows_{\spnode_1}}$ and $(\rho^{\spnode_1}_{\ell})_{\ell=1, \dotsc, \numflows_{\spnode_1}}$ as well as $(\vec{Y}^{\spnode_2}_{\ell})_{\ell=1, \dotsc, \numflows_{\spnode_2}}$ and $(\rho^{\spnode_2}_{\ell})_{\ell=1, \dotsc, \numflows_{\spnode_2}}$ to a new convex decomposition of the given multiflow~$\vec{\mcflow}^{\spnode}$ restricted to the subgraph~$G_{\spnode}$.

In contrast to the $s$-nodes considered before, the situation is more complex for a $p$-node~$\spnode$. By definition of almost unsplittable flows, there is (at most) one commodity in~$\fractionalDemands_{\spnode_1} \cap \fractionalDemands_{\spnode_2}$ that is split between~$\spnode_1$ and~$\spnode_2$.
For simplicity, we assume that~$|\fractionalDemands_{\spnode_1} \cap \fractionalDemands_{\spnode_2}| = 1$, i.e., we assume that a split commodity exists. We refer to this commodity by~$\imiddle$. We also assume that~$\imiddle \notin\{\ileft, \iright\}$. In \Cref{app:special-cases}, we show how to take care of the remaining (simpler) cases where these assumptions do not hold. Finally, we may assume without loss of generality that~$\fractionalDemands_{\spnode_1}=\{\ileft,\imiddle\}$ and~$\fractionalDemands_{\spnode_2}=\{\imiddle,\iright\}$. That is,~$\ileft[\spnode_1]=\ileft$, $\iright[\spnode_1]=\imiddle=\ileft[\spnode_2]$, and~$\iright[\spnode_2]=\iright$, this illustrated in the next example:

\begin{example} \label{example3}
Consider the node~$\spnode_6$ of the $sp$-tree from the original \Cref{example}. Note that~$\spnode_6$ is a $p$-node with children~$\spnode_7$ and~$\spnode_8$. For this node, we have
\begin{align*}
& \vec{z}^{\spnode_6} (\vec{\mcflow}) = (0, \tfrac{3}{8}, 1, 1, 1, \tfrac{1}{4}, 0, 0),
& \fractionalDemands_{\spnode_6} (\vec{\mcflow}) &=  \{2, 6 \},\\
& \vec{z}^{\spnode_7} (\vec{\mcflow}) = (0, \tfrac{3}{8}, 1, \tfrac{1}{2}, 0, 0, 0, 0),
& \fractionalDemands_{\spnode_7} (\vec{\mcflow}) &= \{2, 4 \},\\
& \vec{z}^{\spnode_8} (\vec{\mcflow}) = (0, 0, 0, \tfrac{1}{2}, 1, \tfrac{1}{4}, 0, 0),
& \fractionalDemands_{\spnode_8} (\vec{\mcflow}) &= \{4, 6 \}.
\end{align*}
Therefore, we can denote $\ileft[\spnode_6] = 2=\ileft[\spnode_7]$, $\imiddle[\spnode_6] = 4=\iright[\spnode_7]=\ileft[\spnode_8]$, and $\iright[\spnode_6] = 6=\iright[\spnode_8]$.
\end{example}

With the additional choice of whether to route~$\imiddle$ through either~$\spnode_1$ or $\spnode_2$, we obtain for every set~$J^{\spnode}_j$ in the parent component two additional possibilities for the child nodes. This yields in total eight possible combinations of the routings in a parallel component. %
\begin{table}[t]
	\centering\small
	\renewcommand{\arraystretch}{1.15}
	\begin{tabular}{l>{\color{gray}}c>{\color{gray}}c|l>{\color{gray}}c>{\color{gray}}c|l>{\color{gray}}c>{\color{gray}}c|c}
		\multicolumn{3}{c}{parent~$\spnode$} & \multicolumn{3}{c}{child~$\spnode_1$} & \multicolumn{3}{c}{child~$\spnode_2$} &
		\\
		\multicolumn{1}{c}{routing} & $\mu$ & group &
		\multicolumn{1}{c}{routing} & $\mu$ & group &
		\multicolumn{1}{c}{routing} & $\mu$ & group &
		\\
		\hline
		\multirow{2}{*}{$J^{\spnode}_1 = \emptyset$} &
		\multirow{2}{*}{$\mu^{\spnode}_1$} &
		\multirow{2}{*}{$K^{\spnode}_1$}
		& $J^{\spnode_1}_3 = \{\imiddle\}$ & $\mu^{\spnode_1}_3$ & $K^{\spnode_1}_3$
		&$J^{\spnode_2}_1 = \emptyset$ & $\mu^{\spnode_2}_1$ & $K^{\spnode_2}_1$ &
		$\lambda_1$
		\\
		&&
		& $J^{\spnode_1}_1 = \emptyset$ & $\mu^{\spnode_1}_1$ & $K^{\spnode_1}_1$
		& $J^{\spnode_2}_2 = \{\imiddle\}$ & $\mu^{\spnode_2}_2$ & $K^{\spnode_2}_2$ &
		$\lambda_2$
		\\
		\hline
		\multirow{2}{*}{$J^{\spnode}_2 = \{\ileft\}$} &
		\multirow{2}{*}{$\mu^{\spnode}_2$} &
		\multirow{2}{*}{$K^{\spnode}_2$}
		& $J^{\spnode_1}_4 = \{\ileft, \imiddle\}$ & $\mu^{\spnode_1}_4$ & $K^{\spnode_1}_4$
		& $J^{\spnode_2}_1 = \emptyset$ & $\mu^{\spnode_2}_1$ & $K^{\spnode_2}_1$ &
		$\lambda_3$
		\\
		&&
		& $J^{\spnode_1}_2 = \{\ileft\}$ & $\mu^{\spnode_1}_2$ & $K^{\spnode_1}_2$
		& $J^{\spnode_2}_2 = \{\imiddle\}$ & $\mu^{\spnode_2}_2$ & $K^{\spnode_2}_2$ &
		$\lambda_4$
		\\
		\hline
		\multirow{2}{*}{$J^{\spnode}_3 = \{\iright\}$} &
		\multirow{2}{*}{$\mu^{\spnode}_3$} &
		\multirow{2}{*}{$K^{\spnode}_3$}
		& $J^{\spnode_1}_3 = \{\imiddle\}$ & $\mu^{\spnode_1}_3$ & $K^{\spnode_1}_3$
		& $J^{\spnode_2}_3 = \{\iright\}$ & $\mu^{\spnode_2}_3$ & $K^{\spnode_2}_3$ &
		$\lambda_5$
		\\
		&&
		& $J^{\spnode_1}_1 = \emptyset$ & $\mu^{\spnode_1}_1$ & $K^{\spnode_1}_1$
		& $J^{\spnode_2}_4 = \{\imiddle, \iright\}$ & $\mu^{\spnode_2}_4$ & $K^{\spnode_2}_4$ &
		$\lambda_6$
		\\
		\hline
		\multirow{2}{*}{$J^{\spnode}_4 = \{\ileft, \iright\}$} &
		\multirow{2}{*}{$\mu^{\spnode}_4$} &
		\multirow{2}{*}{$K^{\spnode}_4$}
		& $J^{\spnode_1}_4 = \{\ileft, \imiddle\}$ & $\mu^{\spnode_1}_4$ & $K^{\spnode_1}_4$
		& $J^{\spnode_2}_3 = \{\iright\}$ & $\mu^{\spnode_2}_3$ & $K^{\spnode_2}_3$ &
		$\lambda_7$
		\\
		&&
		& $J^{\spnode_1}_2 = \{\ileft\}$ & $\mu^{\spnode_1}_2$ & $K^{\spnode_1}_2$
		& $J^{\spnode_2}_4 = \{\imiddle, \iright\}$ & $\mu^{\spnode_2}_4$ & $K^{\spnode_2}_4$ &
		$\lambda_8$
	\end{tabular}%
\renewcommand{\arraystretch}{1}%
\caption{Routing options in a parallel component. Every routing option~$J^{\spnode}_j$ of the parent node can be achieved by two combinations of routing options~$J^{\spnode_1}_{j_1}$ and~$J^{\spnode_2}_{j_2}$ in the child nodes.
Each of the eight routings has a weight~$\lambda_i$ that must be consistent with the weights~$\mu^{\spnode_i}_{j_i}$ of the respective routings in the child nodes.}%
\label{tab:parallel:splitting}%
\end{table}
These eight cases are given in \cref{tab:parallel:splitting} showing which groups of flows of the child nodes must be combined and which weights need to be used. For example, the first row states that we need to combine flows from group~$K^{\spnode_1}_3$ of the first child and flows from the group~$K^{\spnode_2}_1$ of the second child, assign a total weight of $\lambda_1$ to these flows and then assign the resulting flows in group~$K^{\spnode}_1$ of the parent node.
\Cref{tab:parallel:splitting} additionally imposes constraints on the weights~$\lambda_i, i=1,\dotsc,8$ of the eight possible routing options in a parallel component. In particular, the weights of any two routing options that correspond to a fixed routing option of the components~$\spnode, \spnode_1,$ or $\spnode_2$ must sum up to the respective~$\mu_j$ value.
For example, the routing option $J^{\spnode_1}_1 = \emptyset$ in the first child~$\spnode$ occurs in the options $i=2$ and $i=6$. Therefore, $\lambda_2 + \lambda_6 = \mu^{\spnode_1}_1$. Overall, we obtain twelve constraints from \cref{tab:parallel:splitting}.
Formally, we want to combine flows from the respective groups and assign a total weight~$\lambda_i$ to all combinations of one row of \cref{tab:parallel:splitting}. The weights~$\lambda_1, \dotsc, \lambda_8 \in [0,1]$ must satisfy~$\sum_{j=1}^8 \lambda_j = 1$ as well as
\begin{align}
	\mu^{\spnode}_1 &= \lambda_1 + \lambda_2, \qquad\qquad
	\mu^{\spnode_1}_1 = \lambda_2 + \lambda_6, \qquad\qquad
	\mu^{\spnode_2}_1 = \lambda_1 + \lambda_3,
	\notag\\
	\mu^{\spnode}_2 &= \lambda_3 + \lambda_4, \qquad\qquad
	\mu^{\spnode_1}_2 = \lambda_4 + \lambda_8, \qquad\qquad
	\mu^{\spnode_2}_2 = \lambda_2 + \lambda_4,
	\notag\\
	\mu^{\spnode}_3 &= \lambda_5 + \lambda_6, \qquad\qquad
	\mu^{\spnode_1}_3 = \lambda_1 + \lambda_5, \qquad\qquad
	\mu^{\spnode_2}_3 = \lambda_5 + \lambda_7,
	\label{eq:linear-system}
	\\
	\mu^{\spnode}_4 &= \lambda_7 + \lambda_8, \qquad\qquad
	\mu^{\spnode_1}_4 = \lambda_3 + \lambda_7, \qquad\qquad
	\mu^{\spnode_2}_4 = \lambda_6 + \lambda_8.
	\notag
\end{align}

We can solve the system~\eqref{eq:linear-system} in a straightforward way. For ease of notation, we introduce the notion of the \emph{second maximum} of a family of numbers $a_1, \dotsc, a_n \in \R$. Formally, we define the second maximum as $\smax (a_1, \dotsc, a_n) \coloneqq a_{\sigma(n-1)}$, where $\sigma$ is a permutation such that $a_{\sigma(i)} \leq a_{\sigma(j)}$ if $i < j$. Notice that $\smax$ denotes the value of the second to last number in the ordering, but not necessarily the second largest of all values. Hence, it can be the case that $\smax(a_1, \dotsc, a_n) = \max\{a_1, \dotsc, a_n\}$ if $a_{\sigma(n-1)} = a_{\sigma(n)}$.
\begin{lemma}\label{lem:lambdas}
	Let $\spnode$ be a $p$-node in the $sp$-tree with child nodes $\spnode_1$ and $\spnode_2$. Denote by $z_{p} \coloneqq z_{\spnode, \ileft}$, $z_{q} \coloneqq z_{\spnode, \iright}$, and $z_{r} \coloneqq z_{\spnode_2, \imiddle}$ and let $M \coloneqq \smax(z_p, z_r, 1 - z_q)$ be the second maximum of these numbers. Then, the numbers
	\begin{align*}
		\lambda_1 &= [(1 - z_q) - M]^+
		&
		\lambda_2 &= [M - z_p]^+
		&
		\lambda_3 &= [M - z_r]^+
		\\
		\lambda_4 &= \min \{z_p, z_r, 1 - z_q\}
		&
		\lambda_5 &= 1 - \max \{z_p, z_r, 1 - z_q\}
		&
		\lambda_6 &= [z_r - M]^+
		\\
		\lambda_7 &= [z_p - M]^+
		&
		\lambda_8 &= [M - (1 - z_q)]^+
	\end{align*}
	satisfy $\lambda_i \in [0, 1]$ and sum up to $1$. Together with the numbers defined in~\eqref{eq:definition:mu} they are a solution to the system~\eqref{eq:linear-system}.
\end{lemma}

Notice that, by definition, for every parallel component~$\spnode$, at most four out of the eight coefficients~$\lambda_i$ are non-zero.

\begin{proof}
	First, we observe that since $z_p, z_q, z_r \in [0, 1]$ we have $0 \leq \lambda_i \leq 1$ for $i=1, \dotsc, 8$ by definition.
	Additionally, we state some basic properties of the positive part and the second maximum. Let $a, b, c \in \R$ and $M = \smax(a, b, c)$. Then,
	\begin{align}
		[a]^+ + [-a]^+ &= |a|, \notag \\
		[a - M]^+ + [M - b]^+ &= [a - b]^+, \notag \\
		[M - a]^+ + \min\{a, b, c\} &= \min\{b, c\}, \notag \\
		[a - M]^+ + 1 - \max\{a, b, c\} &= 1 - \max\{b, c\}
		. \label{prf:lem:lambdas:1}
	\end{align}
	Using the first identity, it is easy to verify that
	\[
		(\lambda_1 + \lambda_8) + (\lambda_2 + \lambda_7) + (\lambda_3 + \lambda_6) = \max\{z_p, z_r, 1-z_q\} -  \min\{z_p, z_r, 1-z_q\}
	\]
	and, thus, $\sum_{i=1}^8 \lambda_i = 1$.

	With the assumption that $\fractionalDemands_\spnode = \{\ileft, \iright\}$ and $\fractionalDemands_{\spnode_1} \cap \fractionalDemands_{\spnode_1} = \{\imiddle\}$ with $\ileft \neq \imiddle \neq \iright$ as well as the assumptions that $\ileft[\spnode_1]=\ileft$, $\iright[\spnode_1]=\imiddle=\ileft[\spnode_2]$, and~$\iright[\spnode_2]=\iright$, we additionally observe that
	\begin{equation}
		z_p = z_{\spnode_1, \ileft[\spnode_1]}, \quad
		z_r = z_{\spnode_2, \ileft} = z_{\spnode_2, \ileft[\spnode_2]} \stackrel{\text{Obs.~\ref{obs:z-in-series-parallel}\ref{obs:z-in-parallel}}}{=} 1 - z_{\spnode_1, \iright[\spnode_1]}, \quad \text{and }
		z_q = z_{\spnode_2, \iright[\spnode_2]}
		.
		\label{prf:lem:lambdas:2}
	\end{equation}
	We can now verify that the numbers~$\lambda_i$ together with the values~$\mu^{\spnode}_j$ defined in~\eqref{eq:definition:mu} are a solution to the linear system given in~\eqref{eq:linear-system}.
	We compute
	\begin{align*}
		\lambda_1 + \lambda_2 &\stackrel{\mathclap{\eqref{prf:lem:lambdas:1}}}{=} [(1 - z_q) - z_p]^+ = [(1 - z_{\spnode, \iright}) - z_{\spnode, \ileft}]^+ = \mu^{\spnode}_1,
		\\
		\lambda_3 + \lambda_4 &\stackrel{\mathclap{\eqref{prf:lem:lambdas:1}}}{=} \min\{z_p, 1 - z_q\} = \min\{ z_{\spnode, \ileft}, 1- z_{\spnode, \iright}\}  = \mu^{\spnode}_2,
		\\
		\lambda_5 + \lambda_6 &\stackrel{\mathclap{\eqref{prf:lem:lambdas:1}}}{=} 1 - \max\{z_p, 1 - z_q\} = 1 - \max\{ z_{\spnode, \ileft}, 1- z_{\spnode, \iright}\}  = \mu^{\spnode}_3,
		\\
		\lambda_7 + \lambda_8 &\stackrel{\mathclap{\eqref{prf:lem:lambdas:1}}}{=} [z_p - (1 - z_q)]^+ = [z_{\spnode, \ileft} - (1 - z_{\spnode, \iright})]^+ = \mu^{\spnode}_4,
	\intertext{and}
		\lambda_2 + \lambda_6 &\stackrel{\mathclap{\eqref{prf:lem:lambdas:1}}}{=} [z_r - z_p]^+ = [z_{\spnode_2, \iright} - z_{\spnode, \ileft}]^+
		\stackrel{\mathclap{\eqref{prf:lem:lambdas:2}}}{=} [(1 - z_{\spnode_1, \iright[\spnode_1]}) - z_{\spnode, \ileft[\spnode_1]}]^+ = \mu^{\spnode_1}_1,
		\\
		\lambda_4 + \lambda_8 &\stackrel{\mathclap{\eqref{prf:lem:lambdas:1}}}{=} \min\{z_p, z_r\} = \min\{ z_{\spnode, \ileft}, z_{\spnode_2, \imiddle}\}
		\stackrel{\mathclap{\eqref{prf:lem:lambdas:2}}}{=} \min\{ z_{\spnode_1, \ileft[\spnode_1]}, 1- z_{\spnode_1, \iright[\spnode_1]}\} = \mu^{\spnode_1}_2,
		\\
		\lambda_1 + \lambda_5 &\stackrel{\mathclap{\eqref{prf:lem:lambdas:1}}}{=} 1 - \max\{z_p, z_r\} = 1 - \max \{ z_{\spnode, \ileft}, z_{\spnode_2, \imiddle}\}
		\stackrel{\mathclap{\eqref{prf:lem:lambdas:2}}}{=} 1 - \max\{ z_{\spnode_1, \ileft[\spnode_1]}, 1- z_{\spnode_1, \iright[\spnode_1]}\}  = \mu^{\spnode_1}_3,
		\\
		\lambda_3 + \lambda_7 &\stackrel{\mathclap{\eqref{prf:lem:lambdas:1}}}{=} [z_r - z_p]^+ = [z_{\spnode, \ileft} - z_{\spnode_2, \iright}]^+
		\stackrel{\mathclap{\eqref{prf:lem:lambdas:2}}}{=} [z_{\spnode, \ileft[\spnode_1]} - (1 - z_{\spnode_1, \iright[\spnode_1]})]^+ = \mu^{\spnode_1}_4.
	\end{align*}
	The last four equations can be checked in exactly the same way.
\end{proof}

For every possible combination~$i=1,\dotsc,8$ of flow groups given in \cref{tab:parallel:splitting}, we can use \cref{alg:linear:combination:refinement} with $\bar{\rho} = \lambda_i$ in order to refine the linear combinations of flows in the respective groups to a new linear combination with flows for the $p$-node~$\spnode$. In this way, we only combine flows that are consistent in the sense that every commodity is either send through~$\spnode_1$ or~$\spnode_2$. Therefore, all flows constructed this way are unsplittable and feasible.
Further, \cref{alg:linear:combination:refinement} ensures that the new flows have a total weight of~$\lambda_i$, which in turn by~\eqref{eq:linear-system} implies that every group~$\flowgroup^{\spnode}_j$ has a total weight of~$\mu^{\spnode}_j$. Therefore, by \cref{lem:mu:consistency}, all flows together form a convex decomposition of the given multiflow.
We illustrate this procedure in the following example.

\begin{example} \label{p-example}
We consider an example for the combination of the convex decompositions in a $p$-node~$\spnode$ with child nodes $\spnode_1$ and $\spnode_2$.
For this example, we assume that we are given the convex decompositions of the child nodes. In particular, we assume that the convex combinations for child~$\spnode_1$ is as follows.
\begin{align*}
{\vec{\mcflow}}^{\spnode_1} &=
\sum_{\ell = 1}^{\numflows_{\spnode_1}} \rho^{\spnode_1}_\ell \vec{Y}^{\spnode_1}_\ell
=
\colorbrace{color1}{
\tfrac{1}{6} \vec{Y}^{\spnode_1}_1
+
\tfrac{1}{4} \vec{Y}^{\spnode_1}_2
}{\substack{\text{group }\flowgroup^{\spnode_1}_1 \\ \mu^{\spnode_1}_1 = \frac{1}{6} + \frac{1}{4} = \frac{5}{12} }}
+
\colorbrace{color2}{
\tfrac{1}{16} \vec{Y}^{\spnode_1}_3 +
\tfrac{3}{16} \vec{Y}^{\spnode_1}_4
}{\substack{\text{group }\flowgroup^{\spnode_1}_2 \\ \mu^{\spnode_1}_2 = \frac{1}{16} + \frac{3}{16} = \frac{1}{4} }}
+
\colorbrace{color3}{
\tfrac{1}{15} \vec{Y}^{\spnode_1}_5 +
\tfrac{2}{15} \vec{Y}^{\spnode_1}_6 +
\tfrac{2}{15} \vec{Y}^{\spnode_1}_7
}{\substack{\text{group }\flowgroup^{\spnode_1}_3 \\ \mu^{\spnode_1}_3 = \frac{1}{15} + \frac{2}{15} + \frac{2}{15} = \frac{1}{3} }}
\end{align*}
The coefficients~$\mu^{\spnode_1}_j$ are consistent with the demand shares that we assume to be $z_p = z_p^{\spnode_1} = \frac{1}{4}$ and $z_q^{\spnode_1} = \frac{1}{3} = 1 - z_r$.
For child~$\spnode_2$, we assume that the convex combinations is as follows.
\begin{align*}
	{\vec{\mcflow}}^{\spnode_2} =
	\sum_{\ell = 1}^{\numflows_{\spnode_2}} \rho^{\spnode_2}_\ell \vec{Y}^{\spnode_2}_\ell
	=
	\colorbrace{color5}{
	\tfrac{3}{10} \vec{Y}^{\spnode_2}_1
	+
	\tfrac{1}{10} \vec{Y}^{\spnode_2}_2
	}{\substack{\text{group }\flowgroup^{\spnode_2}_2 \\ \mu^{\spnode_2}_2 = \frac{3}{10} + \frac{1}{10} = \frac{2}{5} }}
	+
	\colorbrace{color6}{
	\tfrac{1}{4} \vec{Y}^{\spnode_2}_3 +
	\tfrac{1}{12} \vec{Y}^{\spnode_2}_4
	}{\substack{\text{group }\flowgroup^{\spnode_2}_3 \\ \mu^{\spnode_2}_3 = \frac{1}{4} + \frac{1}{12} = \frac{1}{3} }}
	+
	\colorbrace{color7}{
	\tfrac{1}{5} \vec{Y}^{\spnode_2}_5 +
	\tfrac{1}{15} \vec{Y}^{\spnode_2}_6
	}{\substack{\text{group }\flowgroup^{\spnode_2}_4 \\ \mu^{\spnode_2}_4 = \frac{1}{5} + \frac{1}{15} = \frac{4}{15} }}
	\end{align*}
	The coefficients~$\mu^{\spnode_2}_j$ are consistent with the demand shares that we assume to be $z_p^{\spnode_2} = \frac{2}{3}=z_r $ and $z_q = z_q^{\spnode_2} = \frac{3}{5}$.
	Both convex decompositions are depicted in the upper part of \cref{fig:refinements:parallel}.

Now consider the parallel component represented by $\spnode$,  we have  $z_p = \frac{1}{4}$, $z_r = \frac{2}{3}$, and $z_q = \frac{3}{5}$. The second maximum is $M = \smax \big( \frac{1}{4}, \frac{2}{3}, 1 - \frac{3}{5}\big) = \frac{2}{5}$. Plugging this into the formulas from \cref{lem:lambdas}, we obtain
	\begin{align*}
		\lambda_1 = 0,\quad \lambda_2 = \tfrac{3}{20},\quad \lambda_3 = 0,\quad \lambda_4 = \tfrac{1}{4},\quad \lambda_5 = \tfrac{1}{3},\quad
		\lambda_6 = \tfrac{4}{15},\quad  \lambda_7 = \lambda_8 = 0
		.
	\end{align*}

We observe that there are four coefficients~$\lambda_i$ that are non-zero. These coefficients correspond to four possible combinations of routings in the parallel component.
\Cref{tab:parallel:splitting} indicates which flows have to be combined.

For the first non-zero coefficient~$\lambda_i$ with $i = 2$, the group~$\flowgroup^{\spnode}_1$ of the parent node consists of flows resulting from combining the flows in the groups~$\flowgroup^{\spnode_1}_1$ and~$\flowgroup^{\spnode_2}_2$.
We combine the flows and coefficients using \cref{alg:linear:combination:refinement} with $\bar{\rho}_1 = \lambda_2 = \frac{3}{20}$. The output of the algorithm is in this case:
\begin{align*}
	\tilde{\vec{\mcflow}}^{\spnode, 1} &=
	\sum_{\ell=1}^3 \rho^{\spnode}_\ell \vec{Y}^{\spnode}_\ell
	=
	\frac{3}{50} \begin{bmatrix} \vec{Y}^{\spnode_1}_1 \\ \vec{Y}^{\spnode_2}_1  \end{bmatrix} +
	\frac{21}{400} \begin{bmatrix} \vec{Y}^{\spnode_1}_2 \\ \vec{Y}^{\spnode_2}_1  \end{bmatrix} +
	\frac{3}{80} \begin{bmatrix} \vec{Y}^{\spnode_1}_2 \\ \vec{Y}^{\spnode_2}_2  \end{bmatrix}
\intertext{
For $i = 4$, the group~$\flowgroup^{\spnode}_2$ of the parent node consists of flows resulting from combining the flows in the groups~$\flowgroup^{\spnode_1}_2$ and~$\flowgroup^{\spnode_2}_2$.
By applying  \cref{alg:linear:combination:refinement} with $\bar{\rho}_2 = \lambda_4 = \frac{1}{4}$, we obtain the output below:
}
	\tilde{\vec{\mcflow}}^{\spnode, 2} &=
	\sum_{\ell=4}^6 \rho^{\spnode}_\ell \vec{Y}^{\spnode}_\ell
	=
	\frac{1}{16} \begin{bmatrix} \vec{Y}^{\spnode_1}_3 \\ \vec{Y}^{\spnode_2}_1  \end{bmatrix} +
	\frac{2}{16} \begin{bmatrix} \vec{Y}^{\spnode_1}_4 \\ \vec{Y}^{\spnode_2}_1  \end{bmatrix} +
	\frac{1}{16} \begin{bmatrix} \vec{Y}^{\spnode_1}_4 \\ \vec{Y}^{\spnode_2}_2  \end{bmatrix}
\intertext{
Similarly,  for $i = 5$, the group~$\flowgroup^{\spnode}_3$ of the parent node consists of flows resulting from combining the flows in the groups~$\flowgroup^{\spnode_1}_3$ and~$\flowgroup^{\spnode_2}_3$.  Using \cref{alg:linear:combination:refinement}
with $\bar{\rho}_3 = \lambda_5 = \frac{1}{3}$,  the output is as follows:
}
	\tilde{\vec{\mcflow}}^{\spnode, 3} &=
	\sum_{\ell=7}^{10} \rho^{\spnode}_\ell \vec{Y}^{\spnode}_\ell
	=
	\frac{1}{15} \begin{bmatrix} \vec{Y}^{\spnode_1}_5 \\ \vec{Y}^{\spnode_2}_3  \end{bmatrix} +
	\frac{2}{15} \begin{bmatrix} \vec{Y}^{\spnode_1}_6 \\ \vec{Y}^{\spnode_2}_3 \end{bmatrix} +
	\frac{1}{20}  \begin{bmatrix} \vec{Y}^{\spnode_1}_7 \\ \vec{Y}^{\spnode_2}_3 \end{bmatrix} +
	\frac{1}{12}\begin{bmatrix} \vec{Y}^{\spnode_1}_7 \\ \vec{Y}^{\spnode_2}_4  \end{bmatrix}
\intertext{
Finally,   for $i = 6$, the group~$\flowgroup^{\spnode}_4$ of the parent node consists of flows resulting from combining the flows in the groups~$\flowgroup^{\spnode_1}_1$ and~$\flowgroup^{\spnode_2}_4$. We apply \cref{alg:linear:combination:refinement}
with $\bar{\rho}_4 = \lambda_6 = \frac{4}{15}$, and the resulting output is:
}
	\tilde{\vec{\mcflow}}^{\spnode, 4} &=
	\sum_{\ell=11}^{13} \rho^{\spnode}_\ell \vec{Y}^{\spnode}_\ell
	=
\frac{8}{75} \begin{bmatrix} \vec{Y}^{\spnode_1}_1 \\ \vec{Y}^{\spnode_2}_5  \end{bmatrix} +
\frac{7}{75}\begin{bmatrix} \vec{Y}^{\spnode_1}_2 \\ \vec{Y}^{\spnode_2}_5  \end{bmatrix} +
\frac{1}{15} \begin{bmatrix} \vec{Y}^{\spnode_1}_2 \\ \vec{Y}^{\spnode_2}_6  \end{bmatrix}
\end{align*}

Overall, we obtain
\[
	\sum_{\ell=1}^{13} \rho^{\spnode}_\ell \vec{Y}^{\spnode}_\ell
	=
	\tilde{\vec{\mcflow}}^{\spnode, 1} + \tilde{\vec{\mcflow}}^{\spnode, 2} + \tilde{\vec{\mcflow}}^{\spnode, 3} + \tilde{\vec{\mcflow}}^{\spnode, 4}
	=
	\begin{bmatrix}
		\sum_{\ell = 1}^{\numflows_{\spnode_1}} \rho^{\spnode_1}_\ell \vec{Y}^{\spnode_1}_\ell \\
		\sum_{\ell = 1}^{\numflows_{\spnode_2}} \rho^{\spnode_2}_\ell \vec{Y}^{\spnode_2}_\ell
	\end{bmatrix}
	=
	\begin{bmatrix}
		\vec{\mcflow}^{\spnode_1} \\
		\vec{\mcflow}^{\spnode_2}
	\end{bmatrix}
	=
	\vec{\mcflow}^{\spnode}
\]

\end{example}

\subsection{Proof of Theorem~\ref{thm:2dmax}}

Combining \cref{lem:routing:2dmax,lem:recursive:convex} and \cref{thm:almost-unsplittable} completes the proof of \cref{thm:2dmax}.

\begin{proof}[Proof of \cref{thm:2dmax}]
	Let $\vec{\mcflow}$ be a given multiflow with total arc flows~$\vec{\totalflow} = (\totalflow_e)_{e \in E}$. Then, by \cref{thm:almost-unsplittable}, there is another, almost unsplittable flow~$\tilde{\vec{\mcflow}}$ that has the same total flow vector, i.e. $\tilde{\vec{\totalflow}} = \vec{\totalflow}$. \Cref{lem:recursive:convex} implies for $\spnode = \spnode_0$ that there are unsplittable multiflows~$\vec{Y}^{\spnode}_{\ell}, \ell=1,\dotsc,\numflows$ such that
	$\tilde{\vec{\mcflow}} = \sum_{\ell=1}^{\numflows} \rho_{\ell} \vec{Y}_{\ell}$.
	Since the total flow is linear in the multiflow, we also get
	\[
		\vec{\totalflow} = \tilde{\vec{\totalflow}} = \sum_{\ell=1}^{\numflows} \rho_{\ell} \vec{y}_{\ell},
	\]
	where $\vec{y}_{\ell}$ is the total flow of~$\vec{Y}_{\ell}$.
	By \cref{lem:recursive:convex}\emph{(ii)}, every flow $\vec{Y}_{\ell}$ respects some routing option $j \in \{1,2,3,4\}$ in all nodes~$\spnode \in V_T$ of the $sp$-tree. Therefore, \cref{lem:routing:2dmax}\emph{(ii)} implies the desired bound.
\end{proof}

The proof of \Cref{thm:2dmax} can be turned into an efficient algorithm that computes a family of unsplittable flows
together with convex coefficients. The algorithm goes through the $sp$-tree from top to bottom (i.e., from the root to the leafs), successively refining a common convex combination of the flows routed through the components corresponding to the considered tree nodes. Moreover, by regularly applying Carathéodory's Theorem (see, e.g.,~\cite[Chapter~7.7]{Schr86}) to intermediate convex combinations, one can keep the number of unsplittable routings bounded by~$O(k\cdot m)$, where~$m$ is the number of arcs of digraph~$G$. A similar argument is used in the work of Martens et al.~\cite{MartensSalazarSkut06}.

\subsection{Proof of Theorem~\ref{thm:main}}

Finally, we argue that the bounds from \cref{lem:routing:2dmax} can be strengthened by showing that for every component~$\spnode$ and $j \in \{1, \dotsc, 4\}$,
$
	|\componentflow[y]_{\spnode} - \componentflow[x]_{\spnode}| < d_{\max}
	,\text{ for~$j=1,2,3,4$,}
$
holds true, whenever $\mu_j^{\spnode} > 0$. This implies that the difference between the flow through a component for a given routing option~$j=1,2,3,4$ and the fractional flow is at most $d_{\max}$ whenever the corresponding coefficient in the convex combination is non-zero. 

\begin{lemma}
	\label{lem:dmax-bound}
	Let $\vec{\mcflow}$ be an almost unsplittable flow and $\vec{Y}$ an unsplittable flow.
	\begin{enumerate}[(i)]
		\item If $\vec{Y}$ respects the routing option $j \in \{1,2,3,4\}$ in some $\spnode \in V_T$ and $\mu^\spnode_j > 0$ then
		\begin{equation*}
            |\componentflow[y]_{\spnode} - \componentflow[x]_{\spnode}| < d_{\max}
		.
		\end{equation*}
		\item If, for every (leaf) node~$\spnode \in V_T$, $\vec{Y}$ respects some routing option $j_\spnode \in \{1,2,3,4\}$ such that $\mu^{\spnode}_j > 0$, then
		\[
			\totalflow_e - d_{\max} <  y_e < \totalflow_e + d_{\max}
		\]
		for all $e \in E$.
	\end{enumerate}
\end{lemma}

\begin{proof}
	We compute
	$
		|\componentflow[y]_{\spnode} - \componentflow[x]_{\spnode}| =  \big| \sum_{i \in J^{\spnode}_j} d_i - \sum_{i \in \fractionalDemands_{\spnode}} z_{\spnode, i} d_i \big|
	$.
	Assume that $\vec{Y}$ respects some routing option~$j \in \{1,2,3,4\}$ for some $\spnode \in V_T$.\\
	For~$j=1$, we observe that $\mu^{\spnode}_1 > 0$ holds only if $z_{\spnode, \ileft} < 1 - z_{\spnode, \iright}$, which yields
		\begin{align*}
			|\componentflow[y]_{\spnode} - \componentflow[x]_{\spnode}| &= \bigl| - z_{\spnode, \ileft} d_{\ileft} - z_{\spnode, \iright} d_{\iright} \bigr| =  z_{\spnode, \ileft} d_{\ileft} + z_{\spnode, \iright} d_{\iright}
			\leq  ( \underbrace{z_{\spnode, \ileft}}_{\mathclap{< 1 - z_{\spnode, \iright}}}+ z_{\spnode, \iright}) d_{\max} < d_{\max}
			.
		\intertext{For~$j = 2$, observe that}
			|\componentflow[y]_{\spnode} - \componentflow[x]_{\spnode}| &=
			\bigl\vert
			(1 - z_{\spnode, \ileft}) d_{\ileft} -  z_{\spnode, \iright} d_{\iright}
			\bigr\vert \leq \max \{(\underbrace{1 - z_{\spnode, \ileft}}_{<1}) d_{\ileft} , \underbrace{ z_{\spnode, \iright}}_{<1} d_{\iright}  \} < d_{\max},
		\intertext{and, analogously, for $j=3$, we obtain}
			|\componentflow[y]_{\spnode} - \componentflow[x]_{\spnode}| &=
			\bigl\vert
			z_{\spnode, \ileft} d_{\ileft} -  (1 - z_{\spnode, \iright}) d_{\iright}
			\bigr\vert \leq \max \{\underbrace{z_{\spnode, \ileft}}_{<1} d_{\ileft} , (\underbrace{ 1 - z_{\spnode, \iright}}_{<1} ) d_{\iright}  \} < d_{\max}
			.
		\intertext{
		Finally, for $j=4$, we observe that $\mu^{\spnode}_4 > 0$ only if $z_{\spnode, \ileft} > 1 - z_{\spnode, \iright}$, such that
		}
			|\componentflow[y]_{\spnode} - \componentflow[x]_{\spnode}| &= \bigl| (1 - z_{\spnode, \ileft}) d_{\ileft} + (1 - z_{\spnode, \iright}) d_{\iright} \bigr|  = ( 1 - z_{\spnode, \ileft}) d_{\ileft} + (1 - z_{\spnode, \iright}) d_{\iright} \\
			&\leq (1 - z_{\spnode, \ileft} + \underbrace{1 - z_{\spnode, \iright}}_{< z_{\spnode, \ileft}  }) d_{\max} < d_{\max}.
		\end{align*}
		This proves \emph{(i)}.
		For \emph{(ii)} we use that $\componentflow_\spnode = \totalflow_e$ for every leaf node~$\spnode$ that represents an arc~$e$.
	\end{proof}

Finally, this allows us to proof our main result.

\begin{proof}[Proof of \cref{thm:main}]
	We can follow the exact same steps as in the proof of \cref{thm:2dmax}. However, we can use \cref{lem:dmax-bound} instead of \cref{lem:routing:2dmax} to obtain the tighter bound. Notice, that any unsplittable flow respecting a routing option $j \in \{1,2,3,4\}$ in leaf node~$\spnode$ will only contribute to the convex decomposition if also $\mu^{\spnode}_j > 0$. If this is not the case, it may be removed without changing the convex decomposition.
\end{proof}

\section{Special cases for the convex combination}\label{app:special-cases}

In this section we discuss the special cases that may occur when defining the convex decomposition with unsplittable flows.
In \cref{sec:convex:combination} we assumed that $|\fractionalDemands_\spnode| = 2$ for every~$\spnode \in V_T$ and $|\fractionalDemands_{\spnode_1} \cap \fractionalDemands_{\spnode_2}| = 1$ for the two child nodes $\spnode_1$ and $\spnode_2$ of a $p$-node~$\spnode$.
Additionally, we assumed that for every $p$-node with child nodes $\spnode_1, \spnode_2$, all commodities $\ileft, \iright \in \fractionalDemands_{\spnode}$ and $\imiddle \in \fractionalDemands_{\spnode_1} \cap \fractionalDemands_{\spnode_2}$ are distinct.
We will now discuss how to proceed if one or more of these assumptions are not true.

\subsection{Dummy commodities}
If, for some $sp$-node~$\spnode \in V_T$, we have less than two fractionally routed commodities, we use artificial dummy commodities in order to be consistent with all prior definitions.
In this case, we set $\ileft\coloneqq0$ or $\iright\coloneqq0'$, where we interpret~$0$ and~$0'$ as  dummy commodities that are routed through no component of the series-parallel digraph; that is, they satisfy $\mcflow_{0,e}=\mcflow_{0',e}\coloneqq0$ for all~$e \in E$, and thus $z_{\spnode,0}=z_{\spnode,0'}=0$ for all~$\spnode\in V_T$.
Notice that, by this definition, we have $\ileft, \iright \in \fractionalDemands_{\spnode}\cup\{0,0'\}$, and, thus, $z_{\spnode, \ileft}, z_{\spnode, \iright} < 1$.
Using this convention, we may assume that for every~$\spnode \in V_T$ we have two distinct commodities~$\ileft, \iright$ and the four possible routing options $J^{\spnode}_1 = \emptyset, J^{\spnode}_2 = \{\ileft\}, J^{\spnode}_3 = \{\iright\},$ and $J^{\spnode}_4 = \{\ileft, \iright\}$ with weights $\mu^\spnode_j$ as defined in \cref{lem:mu-values}.
With $z_{\spnode,0}=z_{\spnode,0'}=0$ we immediately observe that for $j=1,2,3,4$
\[
	\mu^{\spnode}_j = 0
	\quad \text{whenever} \quad
	0 \in J^{\spnode}_j \text{ or } 0' \in J^{\spnode}_j
	.
\]
This means, whenever a routing option contains a dummy commodity, the corresponding coefficient in the convex combination is zero and, therefore, dummy commodities never occur (with non-zero weight) in the constructed convex combinations.

\subsection{Special cases in parallel components}
Let $\spnode \in V_T$ be a $p$-node with children~$\spnode_1$ and~$\spnode_2$.
By the considerations above, we may assume that~$\ileft$ and $\iright$ are two distinct commodities. However, our construction in \cref{sec:convex:combination} also assumes that there exists a third, distinct commodity $\imiddle \in \fractionalDemands_{\spnode_1} \cap \fractionalDemands_{\spnode_2}$.
This leads to two additional special cases: it may be the case that $\imiddle$ is not distinct from $\ileft$ and $\iright$, and it may also be the case that $\imiddle$ does not exist at all (i.e., $ \fractionalDemands_{\spnode_1} \cap \fractionalDemands_{\spnode_2} = \emptyset$).

\paragraph{Case \texorpdfstring{$\ileft, \iright, \imiddle$}{p, q, r} are not distinct.}

First, we discuss the case where $\imiddle \in \fractionalDemands_{\spnode_1} \cap \fractionalDemands_{\spnode_2}$ exist, but is not distinct from $\ileft, \iright$.
In particular, we consider the case that $\imiddle = \ileft$. (The other case $\imiddle = \iright$ is symmetric.)
In this case, the routing options in \cref{tab:parallel:splitting} are not valid anymore. In contrast, there are now six possible routing options given in \cref{tab:parallel:splitting:case:p=r}.
\begin{table}[t]
	\centering
	\renewcommand{\arraystretch}{1.35}
	\begin{tabular}{l>{\color{gray}}c>{\color{gray}}c|l>{\color{gray}}c>{\color{gray}}c|l>{\color{gray}}c>{\color{gray}}c|c}
		\multicolumn{3}{c}{parent~$\spnode$} & \multicolumn{3}{c}{child~$\spnode_1$} & \multicolumn{3}{c}{child~$\spnode_2$} &
		\\
		\multicolumn{1}{c}{routing} & $\mu$ & group &
		\multicolumn{1}{c}{routing} & $\mu$ & group &
		\multicolumn{1}{c}{routing} & $\mu$ & group &
		\\
		\hline
		$J^{\spnode}_1 = \emptyset$ &
		$\mu^{\spnode}_1$ &
		$K^{\spnode}_1$ &
		$J^{\spnode_1}_1 = \emptyset$ & $\mu^{\spnode_1}_1$ & $K^{\spnode_1}_1$ &
		$J^{\spnode_2}_1 = \emptyset$ & $\mu^{\spnode_2}_1$ & $K^{\spnode_2}_1$ &
		$\lambda_1$
		\\
		\hline
		\multirow{2}{*}{$J^{\spnode}_2 = \{\ileft\}$} &
		\multirow{2}{*}{$\mu^{\spnode}_2$} &
		\multirow{2}{*}{$K^{\spnode}_2$} &
		$J^{\spnode_1}_2 = \{\ileft\}$ & $\mu^{\spnode_1}_2$ & $K^{\spnode_1}_2$ &
		$J^{\spnode_2}_1 = \emptyset$ & $\mu^{\spnode_2}_1$ & $K^{\spnode_2}_1$ &
		$\lambda_3$
		\\
		&&&
		$J^{\spnode_1}_1 = \emptyset$ & $\mu^{\spnode_1}_1$ & $K^{\spnode_1}_1$ &
		$J^{\spnode_2}_2 = \{\ileft\}$ & $\mu^{\spnode_2}_2$ & $K^{\spnode_2}_2$ &
		$\lambda_3$
		\\
		\hline
		$J^{\spnode}_3 = \{\iright\}$ &
		$\mu^{\spnode}_3$ &
		$K^{\spnode}_3$ &
		$J^{\spnode_1}_1 = \emptyset$ & $\mu^{\spnode_1}_1$ & $K^{\spnode_1}_1$ &
		$J^{\spnode_2}_3 = \{\iright\}$ & $\mu^{\spnode_2}_3$ & $K^{\spnode_2}_3$ &
		$\lambda_4$
		\\
		\hline
		\multirow{2}{*}{$J^{\spnode}_4 = \{\ileft, \iright\}$} &
		\multirow{2}{*}{$\mu^{\spnode}_4$} &
		\multirow{2}{*}{$K^{\spnode}_4$} &
		$J^{\spnode_1}_2 = \{\ileft\}$ & $\mu^{\spnode_1}_2$ & $K^{\spnode_1}_2$ &
		$J^{\spnode_2}_3 = \{\iright\}$ & $\mu^{\spnode_2}_3$ & $K^{\spnode_2}_3$ &
		$\lambda_5$
		\\
		&&&
		$J^{\spnode_1}_1 = \emptyset$ & $\mu^{\spnode_1}_1$ & $K^{\spnode_1}_1$ &
		$J^{\spnode_2}_4 = \{\imiddle, \iright\}$ & $\mu^{\spnode_2}_4$ & $K^{\spnode_2}_4$ &
		$\lambda_6$
	\end{tabular}%
	\renewcommand{\arraystretch}{1}
	\caption{Splitting of the commodities in a $p$-node~$\spnode$ with child nodes~$\spnode_1$ and~$\spnode_2$ with the respective sets $J$ and coefficients~$\mu$ in the case $\ileft = \imiddle$.}
	\label{tab:parallel:splitting:case:p=r}
\end{table}%
From this table we obtain the following new linear system replacing the system~\eqref{eq:linear-system}.
\begin{align}
	\mu^{\spnode}_1 &= \lambda_1
	&
	\mu^{\spnode_1}_1 &= \lambda_1 + \lambda_3 + \lambda_4 + \lambda_6
	&
	\mu^{\spnode_2}_1 &= \lambda_1 + \lambda_2
	\notag
	\\
	\mu^{\spnode}_2 &= \lambda_2 + \lambda_3
	&
	\mu^{\spnode_1}_2 &= \lambda_2 + \lambda_5
	&
	\mu^{\spnode_2}_2 &= \lambda_3
	\notag
	\\
	\mu^{\spnode}_3 &= \lambda_4
	&
	\mu^{\spnode_1}_3 &= 0
	&
	\mu^{\spnode_2}_3 &= \lambda_4 + \lambda_5
	\notag
	\\
	\mu^{\spnode}_4 &= \lambda_5 + \lambda_6
	&
	\mu^{\spnode_1}_4 &= 0
	&
	\mu^{\spnode_2}_4 &= \lambda_6
	\label{eq:linear-system:p=r}
\end{align}
Notice that, since $\ileft = \imiddle$, there is only one fractionally routed commodity in the subcomponent~$\spnode_1$. Therefore, within this component we would introduce a dummy commodity~$\iright[\spnode_2] = 0'$, which implies $\mu^{\spnode_1}_3 = \mu^{\spnode_1}_4 = 0$. Also notice that, while in the subcomponent introducing a dummy commodity helps, we can not use the routing options from \cref{tab:parallel:splitting} with a dummy commodity, since the routing options in \cref{tab:parallel:splitting:case:p=r} are fundamentally different.
With the next lemma, we show that there also exist a solution to the system~\eqref{eq:linear-system:p=r}.

\begin{lemma}
	Let $\spnode \in V_T$ be a $p$-node and assume that $\imiddle \in \fractionalDemands_{\spnode_1} \cap \fractionalDemands_{\spnode_2} \neq \emptyset$ as well as $\ileft = \iright$.
	Denote by $z_p \coloneqq z_{\spnode, \ileft}, z_q \coloneqq z_{\spnode, \iright}$ and $z_r \coloneqq z_{\spnode_2, \ileft}$ and let $M \coloneqq \smax(z_p, z_r, 1 - z_q)$ be the second maximum of these numbers. Then the numbers
	\begin{align*}
		\lambda_1 &= [(1-z_q) - z_p]^+,
		&
		\lambda_2 &= M - z_r
		&
		\lambda_3 &= \min\{ z_r, 1 - z_q \},
		\\
		\lambda_4 &= 1 - \max\{ z_p, 1 - z_q \}
		&
		\lambda_5 &=  z_p - M
		&
		\lambda_6 &= [z_r - (1-z_q)]^+
	\end{align*}
	satisfy $\lambda_i \in [0, 1], i=1,\dotsc,6$ and sum up to~$1$. Together with the numbers defined in~\eqref{eq:definition:mu} they are a solution to the system~\eqref{eq:linear-system:p=r}.
\end{lemma}

\begin{proof}
	By \cref{obs:z-in-series-parallel}\ref{obs:z-in-parallel}, we have
	\begin{equation}
		z_r = z_{\spnode_2, \ileft} = z_{\spnode, \ileft} - z_{\spnode_1, \ileft} \leq z_{\spnode, \ileft} = z_p
	\end{equation}
	and, thus, $z_r \leq \max_2(z_p, z_r, 1-z_q) \leq z_p$. This ensures that $\lambda_2, \lambda_5 \geq 0$ and, therefore, $\lambda_i \geq 0, i=1,\dotsc,6$. There are also only three possible orderings of the numbers $z_p, z_r$, and $1-z_q$.
	If $1-z_q \geq z_p \geq z_r$, we compute
	\begin{align*}
		\lambda_2 + \lambda_3 &= z_p - z_r + z_r = z_p = \min\{z_p, 1-z_q\} = \mu^{\spnode}_2
		,\\
		\lambda_4 + \lambda_5 &= 1 - (1 - z_q) + 0 = 1 - \max\{z_r, 1-z_q\} = \mu^{\spnode_2}_3
		;
	\intertext{if $z_p > 1-z_q \geq z_r$, we compute}
		\lambda_2 + \lambda_3 &= (1-z_q) - z_r + z_r = (1-z_q) = \min\{z_p, 1-z_q\}  = \mu^{\spnode}_2
		,\\
		\lambda_4 + \lambda_5 &= 1 - z_p + z_p - (1 - z_q) = 1 - \max\{z_r, 1-z_q\} = \mu^{\spnode_2}_3
		;
	\intertext{and if $z_p \geq z_r > 1-z_q$, we compute}
		\lambda_2 + \lambda_3 &= 0 + (1 - z_q) = \min\{z_p, 1-z_q\} = \mu^{\spnode}_2
		,\\
		\lambda_4 + \lambda_5 &= 1 - z_p + z_p - z_r  = 1 - \max\{z_r, 1-z_q\} = \mu^{\spnode_2}_3
		.
	\end{align*}
	Overall, we have shown $\lambda_2 + \lambda_3 = \mu^{\spnode}_2$ and $\lambda_4 + \lambda_5 = \mu^{\spnode_2}_3$. In addition, we obtain
	\[
		\lambda_2 + \lambda_5 = z_p - z_r = z_{\spnode, \ileft} - z_{\spnode_2, \ileft} = z_{\spnode_1, \ileft} = \min \{z_{\spnode_1, \ileft}, 1 - \underbrace{z_{\spnode_1, 0'}}_{=0} \} = \mu^{\spnode_1}_2
		.
	\]

	Using the identities for the positive part from~\eqref{eq:properties:positive:part}
	we obtain
	\begin{align*}
		\sum_{i=1}^6 \lambda_i &= (\lambda_1 + \lambda_4) + (\lambda_2 + \lambda_5) + (\lambda_3 + \lambda_6)
		= (1 - z_p) + (z_p - r_r) + z_r = 1
		.
	\end{align*}

	The equations $\mu^{\spnode}_1 = \lambda_1, \mu^{\spnode_2}_2 = \lambda_3$, $\mu^{\spnode}_3 = \lambda_4$ and $\mu^{\spnode_2}_4 = \lambda_6$ as well as $\mu^{\spnode_1}_3 = 0$ and $\mu^{\spnode_1}_4 = 0$ are satisfied by the definition of the values $\mu^{\spnode}_j$ in~\eqref{eq:definition:mu}.
	The remaining three equalities of~\eqref{eq:linear-system:p=r} follow since the respective coefficients sum up to~$1$. In particular, we obtain
	\begin{align*}
		\lambda_5 + \lambda_6 &= 1 - \lambda_1 - (\lambda_2 + \lambda_3) - \lambda_4 = 1 - \mu^{\spnode}_1 - \mu^{\spnode}_2 - \mu^{\spnode}_3 = \mu^{\spnode}_4, \\
		\lambda_1 + \lambda_3 + \lambda_4 + \lambda_6 &= 1 - (\lambda_2 + \lambda_5) = 1 - \mu^{\spnode_1}_2 - \mu^{\spnode_1}_3 - \mu^{\spnode_1}_4 = \mu^{\spnode_1}_1, \\
		\lambda_1 + \lambda_2 &= 1 - \lambda_3 - (\lambda_4 + \lambda_5) - \lambda_6 = 1 - \mu^{\spnode_2}_2 - \mu^{\spnode_2}_3 - \mu^{\spnode_2}_4 = \mu^{\spnode_2}_1
	\end{align*}
	which concludes the proof.
\end{proof}

\begin{table}[t]
	\centering
	\renewcommand{\arraystretch}{1.35}
	\begin{tabular}{l>{\color{gray}}c>{\color{gray}}c|l>{\color{gray}}c>{\color{gray}}c|l>{\color{gray}}c>{\color{gray}}c|c}
		\multicolumn{3}{c}{parent~$\spnode$} & \multicolumn{3}{c}{child~$\spnode_1$} & \multicolumn{3}{c}{child~$\spnode_2$} &
		\\
		\multicolumn{1}{c}{routing} & $\mu$ & group &
		\multicolumn{1}{c}{routing} & $\mu$ & group &
		\multicolumn{1}{c}{routing} & $\mu$ & group &
		\\
		\hline
		$J^{\spnode}_1 = \emptyset$ &
		$\mu^{\spnode}_1$ &
		$K^{\spnode}_1$ &
		$J^{\spnode_1}_1 = \emptyset$ & $\mu^{\spnode_1}_1$ & $K^{\spnode_1}_1$ &
		$J^{\spnode_2}_1 = \emptyset$ & $\mu^{\spnode_2}_1$ & $K^{\spnode_2}_1$ &
		$\lambda_1$
		\\
		\hline
		$J^{\spnode}_2 = \{\ileft\}$ &
		$\mu^{\spnode}_2$ &
		$K^{\spnode}_2$ &
		$J^{\spnode_1}_2 = \{\ileft\}$ & $\mu^{\spnode_1}_2$ & $K^{\spnode_1}_2$ &
		$J^{\spnode_2}_1 = \emptyset$ & $\mu^{\spnode_2}_1$ & $K^{\spnode_2}_1$ &
		$\lambda_2$
		\\
		\hline
		$J^{\spnode}_3 = \{\iright\}$ &
		$\mu^{\spnode}_3$ &
		$K^{\spnode}_3$ &
		$J^{\spnode_1}_1 = \emptyset$ & $\mu^{\spnode_1}_1$ & $K^{\spnode_1}_1$ &
		$J^{\spnode_2}_3 = \{\iright\}$ & $\mu^{\spnode_2}_3$ & $K^{\spnode_2}_3$ &
		$\lambda_3$
		\\
		\hline
		$J^{\spnode}_4 = \{\ileft, \iright\}$ &
		$\mu^{\spnode}_4$ &
		$K^{\spnode}_4$ &
		$J^{\spnode_1}_2 = \{\ileft\}$ & $\mu^{\spnode_1}_2$ & $K^{\spnode_1}_2$ &
		$J^{\spnode_2}_3 = \{\iright\}$ & $\mu^{\spnode_2}_3$ & $K^{\spnode_2}_3$ &
		$\lambda_4$
	\end{tabular}%
	\caption{Splitting of the commodities in a $p$-node~$\spnode$ with child nodes~$\spnode_1$ and~$\spnode_2$ with the respective sets $J$ and coefficients~$\mu$ in the case of $\fractionalDemands_{\spnode_1} \cap \fractionalDemands_{\spnode_2} = \emptyset$.}
	\label{tab:parallel:splitting:case:no-r}
\end{table}
\paragraph{Case \texorpdfstring{$\fractionalDemands_{\spnode_1} \cap \fractionalDemands_{\spnode_2} = \emptyset$}{I\_w1 ∩ I\_w2 = ∅}.}
Finally, we consider the case where $\fractionalDemands_{\spnode_1} \cap \fractionalDemands_{\spnode_2} = \emptyset$, i.e. there is no commodity that is split between~$\spnode_1$ and $\spnode_2$. Then, we are just left with the four possible routing options $J^{\spnode}_j, j=1, \dotsc 4$. In particular, there is a one-to-one correspondence between the routing options in the parent node~$\spnode$ and the child nodes depicted in \cref{tab:parallel:splitting:case:no-r}.
We only need to check, that the respective $\mu^{\spnode}_j$-values are consistent.
Since there is no commodity, that is split between the child components, there is also only (at most) one fractionally routed commodity in every child. We therefore use dummy commodities in the child components as follows:
\[
	\ileft[\spnode_1] = \ileft,
	\iright[\spnode_1] = 0'
	\quad\text{and}\quad
	\ileft[\spnode_2] = 0,
	\iright[\spnode_2] = \iright
	.
\]

Using the identities for the positive part from~\eqref{eq:properties:positive:part}, we observe that
\begin{align*}
	\mu^{\spnode}_1 + \mu^{\spnode}_3 &=
	1 - z_{\spnode, \ileft} = [(1 - 0) - z_{\spnode, \ileft}]^+
	= [(1- z_{\spnode, 0'}) - z_{\spnode_1, \ileft[\spnode_1]}]^+ = \mu^{\spnode_1}_1, \\
	\mu^{\spnode}_2 + \mu^{\spnode}_4 &= z_{\spnode, \ileft} = \min\{z_{\spnode, \ileft}, 1\} = \min\{z_{\spnode, \ileft[\spnode_1]}, (1- z_{\spnode, 0'})\} = \mu^{\spnode_1}_2, \\
	\mu^{\spnode}_1 + \mu^{\spnode}_2 &= 1 - z_{\spnode, \iright} = [(1 - z_{\spnode, \iright}) - 0]^+ = [(1 - z_{\spnode_2, \iright[\spnode_2]}) - z_{\spnode_2, 0}]^+ = \mu^{\spnode_2}_1, \\
	\mu^{\spnode}_3 + \mu^{\spnode}_4 &= 1 - (1 - z_{\spnode, \iright})
	= 1 - \max\{ z_{\spnode_2, 0}, 1 - z_{\spnode, \iright} \} = \mu^{\spnode_2}_3
\end{align*}
proving that the respective $\mu^{\spnode}_j$ values are consistent.


\subparagraph{Acknowledgements} The authors thank the anonymous referees of the extended abstract~\cite{MSW-IPCO25} for their insightful comments, which helped improve the presentation of this paper. The second author is especially grateful to Chandra Chekuri and Bruce Shepherd for valuable discussions on cut conditions for multiflows during the 2024 Oberwolfach Workshop on Combinatorial Optimization, which motivated the result presented in \Cref{thm:strengthened-cut-condition}.

\bibliographystyle{plain} 
\bibliography{submission}

\appendix

\section{Refinements of convex combinations}\label{app:refinements}

Our construction of the convex decompositions used in \cref{thm:2dmax,thm:main} relies on a recursive approach. We first construct convex combinations of the flows on single arcs (i.e. for leaf vertices of the $sp$-tree) and combine them recursively in the $s$- and $p$-nodes of the $sp$-tree. In this subsection, we describe the formal process that we refer to as \emph{refinement of convex combinations} used to combine convex combinations in an intuitive and straightforward way.

In this section, we assume we are given two vectors $\vec{v} \in \R^n$ and $\vec{w} \in \R^m$ that are convex combinations of vectors $\vec{v}_\ell \in \R^n, \ell=1, \dotsc, L_1$ and $\vec{w}_\ell \in \R^m, \ell=1, \dotsc, L_2$, i.e., there are coefficients $\lambda_\ell \in [0,1], \ell=1, \dotsc, L_1$ and $\mu_\ell \in [0,1], \ell=1, \dotsc, L_2$ such that
\[
	\vec{v} = \sum_{\ell=1}^{L_1} \lambda_\ell \vec{v}_\ell
	\qquad\text{and}\qquad
	\vec{w} = \sum_{\ell=1}^{L_2} \mu_\ell \vec{w}_\ell
\]
as well as $\sum_{\ell=1}^{L_1} \lambda_\ell = 1$ and $\sum_{\ell=1}^{L_2} \mu_\ell = 1$.
We are interested in finding a convex decomposition of the vector $\vec{u} = [\vec{v}^{\top}, \vec{w}^{\top}]^{\top} \in \R^{n+m}$, i.e., we want to find vectors $\vec{u}_\ell \in \R^{n+m}$ and coefficients $\rho_\ell \in [0,1]$ such that
\[
	\begin{bmatrix} \vec{v} \\ \vec{w} \end{bmatrix}
	=
	\sum_{\ell = 1}^{L_3} \rho_{\ell} \vec{u}_{\ell}
\]
and $\sum_{\ell=1}^{L_3} \rho_{\ell} = 1$. We use the following procedure in order to construct the convex decomposition.

\begin{algorithm}[H]
	\SetKwInOut{Input}{input}
	\SetKwInOut{Output}{output}
	\Input{convex decompositions $\vec{v} = \sum_{\ell=1}^{L_1} \lambda_\ell \vec{v}_\ell$ and $\vec{w} = \sum_{\ell=1}^{L_2} \mu_\ell \vec{w}_\ell$}
	\Output{convex decomposition $\begin{bmatrix} \vec{v} \\ \vec{w} \end{bmatrix} = \sum_{\ell = 1}^{L_3} \rho_{\ell} \vec{u}_{\ell}$}
	Initialize $i \leftarrow 1$, $j \leftarrow 1$, $\ell \leftarrow 1$\;
	\While{$i \leq L_1$ and $j \leq L_2$}{
		Set $\vec{u}_\ell \leftarrow \begin{bmatrix} \vec{v}_i \\ \vec{w}_j \end{bmatrix}$\quad\text{and}\quad$\rho_\ell = \min \{\lambda_i, \mu_j\}$\;
		Set $\lambda_i \leftarrow \lambda_i - \rho_\ell$\quad{}and\quad$\mu_j \leftarrow \mu_j - \rho_{\ell}$\;
		If $\lambda_i = 0$ set $i \leftarrow i + 1$;\quad{}if $\mu_j = 0$ set $j \leftarrow j + 1$\;
		$\ell \leftarrow \ell + 1$\;
	}

	\caption{Refinement of two convex decompositions.}
	\label{alg:convex:refinement}
\end{algorithm}

\cref{fig:convex:refinement} shows the output of \cref{alg:convex:refinement} for a specific example. In the following lemma, we prove the correctness of the algorithm and show that its runtime as well as the output size is linear in the input.

\begin{lemma}
	\cref{alg:convex:refinement} is correct. Furthermore, it has a time complexity of $\mathcal{O}(L_1 + L_2)$ and the convex decomposition constructed by the algorithm consists of $L_3 \leq L_1 + L_2 - 1$ many vectors.
\end{lemma}

\begin{proof}
	First, we observe that in every iteration of \cref{alg:convex:refinement} $i$ or $j$ (or both) is increased by exactly~$1$. Therefore, the algorithm terminates after at most $L_1 + L_2 - 1$ iterations. This is also the number of vectors in the output.

	Further, it es easy to see that the following invariant holds true.
	At the end of every iteration, the total amount of by how much the coefficients $\lambda_i$ have been reduced is the same as how much the coefficients $\mu_j$ have been reduced and, additionally, this is exactly the sum of all $\rho_{\ell}$ defined so far.
	Since both the $\lambda_i$s and the $\mu_j$s sum to exactly~$1$, this implies that in the very last step both $\lambda_i$ and $\mu_j$ are reduced to~$0$ and, in particular, at the end all $\lambda_i$ and $\mu_j$ are equal to $0$ (i.e., all mass of both convex combination was distributed).
	Also, by construction, for all $\ell$ for which $\vec{u}_\ell$ contains some fixed vector $\vec{v}_i$, the corresponding coefficients~$\rho_{\ell}$ sum exactly to~$\lambda_i$. The same is true for vectors $\vec{w}_j$ and coefficients~$\mu_j$. This guarantees that the output is indeed a convex decomposition.
\end{proof}

\def\convboxwidth{0.77*\linewidth/1cm}
\def\convboxheight{1.5em}
\begin{figure}[t]

	\newcommand{\combvector}[2]{
		\begin{bmatrix}
			\vec{v}_{#1} \\ \vec{w}_{#2}
		\end{bmatrix}
	}

	\newcommand{\drawconvexcomb}[1][\convboxheight]{
		\fill[\col!80, draw=black, thick] (\start, -#1/2) rectangle ++(\lam, #1);
	}
	\newcommand{\convexcomblabels}[1][\convboxheight]{
		\draw[thick] (\start, {#1 /2}) -- ++(0, - #1);
		\draw[decorate, decoration={brace, mirror}, \col] (\start, -#1) -- ++(\lam, 0) node[midway, below=.5ex] {\scriptsize$\lbl$};
	}
	\newcommand{\vectorlabels}[1][0]{
		\node[\col] at ({\start+\lam/2}, #1) {\color{white}\footnotesize$\veclbl$};
	}
	\newcommand{\vectorlabelsabove}[1][0]{
		\node[\col, above] at ({\start+\lam/2}, #1) {\footnotesize$\veclbl$};
	}

	\def\firstcc{%
		{0}/{1/4}/{\lambda_1 = \frac{1}{4}}/{\vec{v}_1}/color1,%
		{1/4}/{1/2}/{\lambda_2 = \frac{1}{2}}/{\vec{v}_2}/color2,%
		{3/4}/{1/4}/{\lambda_3 = \frac{1}{4}}/{\vec{v}_3}/color3%
	}

	\def\secondcc{%
		{0}/{1/6}/{\mu_1 = \frac{1}{6}}/{\vec{w}_1}/color4,%
		{1/6}/{1/3}/{\mu_2 = \frac{1}{3}}/{\vec{w}_2}/color5,%
		{1/2}/{1/3}/{\mu_3 = \frac{1}{3}}/{\vec{w}_3}/color6,%
		{5/6}/{1/6}/{\mu_4 = \frac{1}{6}}/{\vec{w}_4}/color7%
	}

	\begin{center}
		\begin{tikzpicture}
			\begin{scope}[xscale=\convboxwidth]
				\node[anchor=west] at (-.02, 1) {\small\textbf{Input:}};

				\foreach \start/\lam/\lbl/\veclbl/\col in \firstcc {
					\drawconvexcomb
					\convexcomblabels
					\vectorlabels
				}
				\draw[very thick] (0,-\convboxheight/2) -- (0, \convboxheight/2) node[above] {\footnotesize $0$};
				\draw[very thick] (1,-\convboxheight/2) -- (1, \convboxheight/2) node[above] {\footnotesize $1$};
				\node[right=.5em] at (1,0) {$\displaystyle \vec{v} = \sum_{i=1}^3 \lambda_i \vec{v}_i$};
			\end{scope}

			\begin{scope}[xscale=\convboxwidth,shift={(0,-2)}]
				\foreach \start/\lam/\lbl/\veclbl/\col in \secondcc {
					\drawconvexcomb
					\convexcomblabels
					\vectorlabels
				}
				\draw[very thick] (0,-\convboxheight/2) -- (0, \convboxheight/2) node[above] {\footnotesize $0$};
				\draw[very thick] (1,-\convboxheight/2) -- (1, \convboxheight/2) node[above] {\footnotesize $1$};
				\node[right=.5em] at (1,0) {$\displaystyle \vec{w} = \sum_{i=1}^4 \lambda_i \vec{w}_i$};
			\end{scope}

			\begin{scope}[xscale=\convboxwidth,shift={(0,-5.666)}]
				\node[anchor=west] at (-.02, 1.8) {\small\textbf{Output:}};

				\begin{scope}[shift={(0,\convboxheight/4)}]
				\foreach \start/\lam/\lbl/\veclbl/\col in \firstcc {
					\drawconvexcomb[\convboxheight/2]
				}
				\end{scope}
				\begin{scope}[shift={(0,-\convboxheight/4)}]
				\foreach \start/\lam/\lbl/\veclbl/\col in \secondcc {
					\drawconvexcomb[\convboxheight/2]
				}
				\end{scope}

				\foreach \start/\lam/\lbl/\veclbl/\col in {%
					{0}/{1/6}/{\rho_1 = \frac{1}{6}}/{\vec{u}_1 {=}\! \combvector{1}{1}}/black,%
					{1/6}/{1/12}/{\rho_2 = \frac{1}{12}}/{~~\vec{u}_2 {=}\! \combvector{1}{2}}/black,%
					{1/4}/{1/4}/{\rho_3 = \frac{1}{4}}/{\vec{u}_3 {=}\! \combvector{2}{2}}/black,%
					{1/2}/{1/4}/{\rho_4 = \frac{1}{4}}/{\vec{u}_4 {=}\! \combvector{2}{3}}/black,%
					{3/4}/{1/12}/{\rho_5 = \frac{1}{12}}/{\vec{u}_5 {=}\! \combvector{3}{3}}/black,%
					{5/6}/{1/6}/{\rho_6 = \frac{1}{6}}/{\vec{u}_6 {=}\! \combvector{3}{4}}/black%
				} {
					\convexcomblabels[3.5ex]
					\vectorlabelsabove[2ex]
				}

				\draw[very thick] (0,-2.5ex) -- (0, 2.5ex) node[above] {\footnotesize $0$};
				\draw[very thick] (1,-2.5ex) -- (1, 2.5ex) node[above] {\footnotesize $1$};

				\node[right=.5em] at (1,0) {$\displaystyle \begin{bmatrix} \vec{v} \\ \vec{w} \end{bmatrix}\! = \sum_{i=1}^6 \rho_i \vec{u}_i$};
			\end{scope}

		\end{tikzpicture}
	\end{center}

	\caption{The refinement of two convex decompositions. Above are the input convex decompositions. The colors represent the different vectors and the lengths of the line segments represent the convex coefficients. Below the output of \cref{alg:convex:refinement} is shown for this input. Every line segment consists of two colors representing a combination of a vector $\vec{v}_i$ and vector $\vec{w}_j$.}
	\label{fig:convex:refinement}
\end{figure}

\begin{algorithm}[t]
	\SetKwInOut{Input}{input}
	\SetKwInOut{Output}{output}
	\Input{Linear combinations $\displaystyle \vec{v} = \sum_{\ell=1}^{L_1} \lambda_\ell \vec{v}_\ell$ and $\displaystyle \vec{w} = \sum_{\ell=1}^{L_2} \mu_\ell \vec{w}_\ell$ \\ with $\lambda_\ell, \mu_\ell \geq 0$ and $\bar{\lambda} \coloneqq \sum_{\ell=1}^{L_1} \lambda_{\ell}$ and $\bar{\mu} \coloneqq \sum_{\ell=1}^{L_2} \mu_{\ell}$;\vspace{.5ex}\\
	a positive number $\bar{\rho} > 0$.} \vspace{.5ex}
	\Output{linear combination $\displaystyle \begin{bmatrix} \frac{\bar{\rho}}{\bar{\lambda}} \vec{v} \\[1ex] \frac{\bar{\rho}}{\bar{\mu}} \vec{w} \end{bmatrix} = \sum_{\ell = 1}^{L_3} \rho_{\ell} \vec{u}_{\ell}$ with $\displaystyle\sum_{\ell = 1}^{L_3} \rho_\ell = \bar{\rho}$}
	Set $\tilde{\lambda}_\ell = \lambda_\ell / \bar{\lambda}, \ell = 1, \dotsc, L_1$ and $\tilde{\mu}_\ell = \mu_\ell / \bar{\mu}, \ell = 1, \dotsc, L_2$\;
	Run \cref{alg:convex:refinement} with normalized coefficients $\tilde{\lambda}_\ell$ and $\tilde{\mu}_\ell$ with output vectors $\vec{u}_\ell$ and coefficients $\tilde{\rho}_{\ell}$\;
	Set $\rho_{\ell} = \bar{\rho} \, \tilde{\rho}_{\ell}$\;
	\caption{Refinement of two non-negative linear combinations.}
	\label{alg:linear:combination:refinement}
\end{algorithm}

We note that \cref{alg:convex:refinement} can be easily generalized to arbitrary positive linear combinations, i.e. settings where the coefficients $\lambda_i, \mu_j$ and~$\rho_{\ell}$ do not sum up to~$1$. In this case, the coefficients $\lambda_i$ and~$\mu_j$ can be rescaled before executing the algorithm. Rescaling the coefficients~$\rho_\ell$ at the end allows for coefficients that sum to an arbitrary positive value. This procedure is summarized in~\cref{alg:linear:combination:refinement}.

\end{document}